\documentclass[a4paper,reqno,11pt]{amsart}
\usepackage{amssymb}
\usepackage{amstext}
\usepackage{amsmath}
\usepackage{amscd}
\usepackage{amsthm}
\usepackage{amsfonts}
\usepackage{enumerate}
\usepackage{graphicx}
\usepackage{latexsym}
\usepackage{mathrsfs}
\input xy
\xyoption{all}
\usepackage{lscape}
\usepackage{pstricks}
\usepackage{url}
\newtheorem{theorem}{Theorem}[section]
\newtheorem{corollary}[theorem]{Corollary}
\newtheorem{lemma}[theorem]{Lemma}
\newtheorem{definition-theorem}[theorem]{Definition-Theorem}
\newtheorem{definition-proposition}[theorem]{Definition-Proposition}
\newtheorem{proposition}[theorem]{Proposition}

\theoremstyle{definition}
\newtheorem{definition}[theorem]{Definition}
\newtheorem{remark}[theorem]{Remark}
\newtheorem{example}[theorem]{Example}

\newtheorem{question}[theorem]{Question}

\newtheorem{notation}[theorem]{Notation}
\numberwithin{equation}{section}
\newcommand{\add}{\mathsf{add}\hspace{.01in}}
\newcommand{\Fac}{\mathsf{Fac}\hspace{.01in}}
\newcommand{\inj}{\mathsf{inj}\hspace{.01in}}
\renewcommand{\mod}{\mathsf{mod}\hspace{.01in}}
\newcommand{\proj}{\mathsf{proj}\hspace{.01in}}
\newcommand{\Sub}{\mathsf{Sub}\hspace{.01in}}
\newcommand{\cok}{\operatorname{Cok}\nolimits}
\newcommand{\codim}[1]{{#1}\operatorname{-codim}}
\newcommand{\End}{\operatorname{End}\nolimits}
\newcommand{\Ext}{\operatorname{Ext}\nolimits}
\newcommand{\gl}{\operatorname{gldim}\nolimits}
\renewcommand{\dim}[1]{{#1}\operatorname{-dim}}
\newcommand{\dom}{\operatorname{domdim}\nolimits}

\newcommand{\gdom}[1]{{#1}\operatorname{-domdim}}
\newcommand{\Hom}{\operatorname{Hom}\nolimits}
\newcommand{\id}{\operatorname{id}\nolimits}
\newcommand{\im}{\operatorname{Im}\nolimits}
\renewcommand{\ker}{\operatorname{Ker}\nolimits}
\newcommand{\op}{\operatorname{op}\nolimits}
\newcommand{\pd}{\operatorname{pd}\nolimits}
\newcommand{\tr}{\operatorname{Tr}\nolimits}
\newcommand{\tilt}{\mbox{\rm tilt}\hspace{.01in}}
\newcommand{\xto}{\xrightarrow}
\newcommand{\kD}{\mathrm{D}}
\begin{document}
\title[Tilting modules and dominant dimension]{Tilting modules and dominant dimension with respect to injective modules}

\author{Takahide Adachi}\address{T.~Adachi: Faculty of Global and Science Studies, Yamaguchi University, 1677-1 Yoshida, Yamaguchi 753-8541, Japan}\email{tadachi@yamaguchi-u.ac.jp}\thanks{T.~Adachi is supported by JSPS KAKENHI Grant Number JP17J05537 and 20K14291.}
\author{Mayu Tsukamoto}\address{M.~Tsukamoto: Graduate school of Sciences and Technology for Innovation, Yamaguchi University, 1677-1 Yoshida, Yamaguchi 753-8512, Japan}\email{tsukamot@yamaguchi-u.ac.jp}\thanks{M.~Tsukamoto is supported by JSPS KAKENHI Grant Number JP19K14513.}

\begin{abstract}
In this paper, we study a relationship between tilting modules with finite projective dimension and dominant dimension with respect to injective modules as a generalization of results of Crawley-Boevey--Sauter, Nguyen--Reiten--Todorov--Zhu and Pressland--Sauter. Moreover, we give characterizations of almost $n$-Auslander--Gorenstein algebras and almost $n$-Auslander algebras by the existence of tilting modules. As an application, we describe a sufficient condition for almost $1$-Auslander algebras to be strongly quasi-hereditary by comparing such tilting modules and characteristic tilting modules.
\end{abstract}
\maketitle

\section{Introduction}
Tilting theory gives a universal method to construct derived equivalences and is considered as one of the effective tools in the study of many areas of mathematics (for example, the representation theories of finite dimensional algebras, finite groups and algebraic groups, algebraic geometry, and algebraic topology).
In this theory, the notion of tilting modules plays a crucial role.
More precisely, the endomorphism algebras of tilting modules are derived equivalent to the original algebra.
Hence it is important to give a construction of tilting modules for a given algebra.

In \cite{CBS17}, Crawley-Boevey--Sauter give a new characterization of which artin algebras with global dimension at most two are Auslander algebras via the existence of certain tilting modules contained in $\Fac_{1}(I)\cap \Sub^{1}(I)$, where $I$ is a maximal projective-injective direct summand of $A$. 
Here for each $i\ge 1$ and an $A$-module $Q$, $\Sub^{i}(Q)$ is the full subcategory of $\mod A$ whose objects are those $X$ for which there is an exact sequence
\begin{align}
0 \to X \to Q^0 \to Q^1 \to \cdots \to Q^{i-1} \notag
\end{align}
such that $Q^{j} \in \add Q$ for each $0 \le j \le i-1$. Dually, we define $\Fac_{i}(Q)$ (see Definition \ref{def-subfac}).
As a refinement, Nguyen--Reiten--Todorov--Zhu show the following theorem.

\begin{theorem}[{\cite[Lemma 1.1]{CBS17}} and {\cite[Theorem 3.3.4]{NRTZ19}}]
Let $A$ be a non-self-injective artin algebra and $I$ a maximal projective-injective direct summand of $A$.
Then $\dom A \ge 2$ if and only if there exists a unique basic tilting module  such that its projective dimension is exactly one and it is contained in $\Fac_{1}(I)\cap \Sub^{1}(I)$.
\end{theorem}

In particular, since an artin algebra $A$ is an Auslander algebra if and only if it satisfies $\gl A \le 2 \le \dom A$, we obtain a new characterization of artin algebras to be Auslander algebras by the existence of certain tilting modules.
Furthermore, Pressland--Sauter \cite{PrSa} characterize minimal $n$-Auslander--Gorenstein algebras (that is, $\id A \le n+1 \le \dom A$) and $n$-Auslander algebras (that is, $\gl A \le n+1 \le \dom A$) by using tilting modules with finite projective dimension. For details of these algebras, see \cite{IS18} and \cite{I07} respectively.

In this paper, we give a generalization of their results. 
Our starting point of this study is to give a relative version of their theorem.
Namely, we study a relationship between tilting modules and relative Auslander algebras in the sense of Iyama.
A relative Auslander algebra is realized as the endomorphism algebra of an additive generator of a faithful torsion-free class of an artin algebra.
Since the finitely generated module category of a representation-finite artin algebra is a faithful torsion-free class with additive generator, relative Auslander algebras are one of the generalizations of Auslander algebras. 
Moreover, Iyama gives a homological interpretation by a generalization of dominant dimension.
For an injective $A$-module $I$, we write $\gdom{I}A \ge n+1$ if $A\in\Sub^{n+1}(I)$.
If $I$ is a maximal projective-injective direct summand of $A$, then $\gdom{I}A=\dom A$. 

\begin{theorem}[{\cite[Theorem 2.1]{I105}}]\label{I105-thm2.1}
Let $A$ be an artin algebra.
Let $I$ $($respectively, $J$$)$ be a basic injective module with the property that $\add I$ $($respectively, $\add J$$)$ consists of all right $($respectively, left$)$ injective $A$-modules with projective dimension at most one.
Then $A$ is a relative Auslander algebra if and only if it satisfies $\gl A \le 2 \le  \min\{ \gdom{I}A, \gdom{J}A^{\op}\}$.
\end{theorem}

From the viewpoint of Theorem \ref{I105-thm2.1}, we introduce the notions of almost Auslander--Gorenstein algebras and almost Auslander algebras, which are also generalizations of minimal Auslander--Gorenstein algebras and Auslander algebras respectively.
Let $I$ be as in Theorem \ref{I105-thm2.1}.
We call an algebra $A$ an \emph{almost} \emph{$n$-Auslander--Gorenstein algebra} (respectively, \emph{almost $n$-Auslander algebra}) if it satisfies 
\begin{align}
\id A \le n+1 \le \gdom{I} A\ \textnormal{(respectively, $\gl A \le n+1 \le \gdom{I} A$)}. \notag
\end{align}
Note that if $\add I = \proj A \cap \inj A$, then these definitions coincide with those of minimal $n$-Auslander--Gorenstein algebras and $n$-Auslander algebras respectively.

To give a characterization of these algebras by tilting modules, we start with studying a connection between tilting modules with finite projective dimension and dominant dimension with respect to injective modules with projective dimension at most one. 
The following theorem is one of main results in this paper.

\begin{theorem}[Theorem \ref{main-thm1}]\label{intro-thm}
Fix an integer $n\ge 0$.
Let $A$ be an artin algebra and $I$ an injective module with projective dimension at most one. 
Then $\gdom{I} A \ge n+1$ if and only if there exists a unique basic tilting module such that its projective dimension is exactly $d$ and it is contained in $\Fac_{d}(I) \cap \Sub^{n+1-d}(I)$ for an integer $0\le d \le \min\{ \id A, n+1\}$.
\end{theorem}

As an application, we give characterizations of almost $n$-Auslander--Gorenstein algebras and almost $n$-Auslander algebras, which is a refinement of \cite[Lemma 1.3]{HU96} for Iwanaga--Gorenstein algebras. 

\begin{theorem}[Theorem \ref{thm-almost-AG} and Proposition \ref{prop-id1}]\label{intro-thm2}
Let $A$ be a non-self-injective artin algebra and $n\ge 1$ an integer.
Let $I$ be as in Theorem \ref{I105-thm2.1}.
Then the following statements are equivalent. 
\begin{itemize}
\item[(1)] $A$ is an almost $n$-Auslander--Gorenstein algebra. 
\item[(2)] For an integer $1\le d \le \min\{ \id A, n+1\}$, there exists a unique basic tilting module such that 
\begin{itemize}
\item[(a)] its projective dimension is exactly $d$,
\item[(b)] it is contained in $\Fac_{d}(I) \cap \Sub^{n+1-d}(I)$, and
\item[(c)] it is cotilting with injective dimension exactly $n+1-d$ if $\id A>1$ and $0$ if $\id A =1$.
\end{itemize}
\end{itemize}
If in addition we assume $\gl A <\infty$, then the following statement is also equivalent.
\begin{itemize}
\item[(3)] $A$ is an almost $n$-Auslander algebra. 
\end{itemize}
\end{theorem}

Note that we can recover Crawley-Boevey--Sauter's result and more generally Pressland--Sauter's result (see Corollary \ref{PrSa-Theorem1}).
Furthermore, we obtain a characterization of relative Auslander algebras in terms of such tilting modules by combining Theorem \ref{intro-thm2} and its dual statement.
As a new perspective, we explain the construction of the tilting modules in Theorems \ref{intro-thm} and \ref{intro-thm2} from the viewpoint of tilting mutation theory (see Proposition \ref{prop-construct-tilting}(1)).

Next we study a relationship between almost $1$-Auslander algebras and strongly quasi-hereditary algebras which are a special class of quasi-hereditary algebras. 
Quasi-hereditary algebras arise from the representation theories of complex Lie algebras and algebraic groups. 
One of the important properties of quasi-hereditary algebras is the existence of tilting modules, called characteristic tilting modules, by Ringel \cite{R91}.  
Recall that strongly quasi-hereditary algebras are defined as quasi-hereditary algebras whose standard modules have projective dimension at most one and costandard modules have injective dimension at most one.  
It is known that if an artin algebra is strongly quasi-hereditary, then its global dimension is at most two \cite[Proposition A.2]{R10}. 
However, the converse does not hold in general.
By focusing on the connection between the tilting modules in Theorem \ref{intro-thm2} and characteristic tilting modules, we give a sufficient condition for almost $1$-Auslander algebras to be strongly quasi-hereditary algebras. 

\begin{theorem}[Theorem \ref{thm-tilting-sqh} and Corollary \ref{cor-sqh-Aus}]
Let $A$ be an almost $1$-Auslander algebra.
Let $\mathbb{T}^{1}$ be the tilting module with projective dimension exactly one in Theorem \ref{intro-thm2} and $\mathbb{T}$ the characteristic tilting module of $A$. 
If $\mathbb{T}$ coincides with $\mathbb{T}^{1}$, then $A$ is a strongly quasi-hereditary algebra. 
Moreover, if $A$ is an Auslander algebra, then the converse also holds.
\end{theorem}

\subsection*{Notation}
Throughout this paper, $A$ is an artin algebra and $\kD$ is the Matlis dual functor.
For simplicity, we assume that $A$ is non-semisimple and basic.
We denote by $\gl A$ the global dimension of $A$ and $\dom A$ the dominant dimension of $A$.
We write $\mod A$ for the category of finitely generated right $A$-modules and $\proj A$ (respectively, $\inj A$) for the full subcategory of $\mod A$ consisting of projective (respectively, injective) $A$-modules.
For $M \in \mod A$, we denote by $\add M$ the full subcategory of $\mod A$ whose objects are direct summands of finite direct sums of $M$.
We denote by $\pd M$ (respectively, $\id M$) the projective (respectively, injective) dimension of $M$.

\section{Dominant dimension with respect to injective modules}
In this section, we recall the definition of dominant dimension with respect to injective modules (see \cite{CX16}, \cite{I105} and \cite{I205} for details) and  collect related results.

Let $A$ be an artin algebra.
We start this section by recalling the following notions.
A morphism $f: X \to Y$ is said to be \emph{left minimal} if each morphism $g:Y \to Y$ with $gf=f$ is an automorphism. 
Let $Q$ be an $A$-module. A morphism $f:X\to Y$ is called a \emph{left $\add Q$-approximation} of $X$ if $Y\in \add Q$ and $\Hom(f,Q)$ is an epimorphism.
Moreover, a left $\add Q$-approximation is said to be \emph{minimal} if it is left minimal.
Dually, we define a \emph{right minimal morphism}, a \emph{right $\add Q$-approximation} and a \emph{minimal right $\add Q$-approximation}. 

Throughout this paper, the following notation is convenient. 

\begin{definition}\label{def-subfac}
Fix an integer $n \ge 0$.
Let $A$ be an artin algebra and $Q$ an $A$-module. 
\begin{itemize}
\item[(1)] We define $\Sub^{n+1}(Q)$ to be the full subcategory of $\mod A$ whose objects are those $X$ for which there is an exact sequence
\begin{align}
0 \to X \xto{f^{0}} Q^0 \xto{f^{1}} Q^1 \to \cdots \xto{f^{n}} Q^{n} \notag
\end{align}
such that $Q^{i} \in \add Q$ for each $0 \le i \le n$.
Let $\Sub^{0}(Q):=\mod A$. 
Moreover, we write $\codim{Q}X= n$ if
$Q^{i}\neq 0$, the inclusion $\im f^{i}\to Q^{i}$ is left minimal for each $0\le i\le n$, and $f^{n}$ can be chosen to be an epimorphism.
\item[(2)] We define $\Fac_{n+1}(Q)$ to be the full subcategory of $\mod A$ whose objects are those $Y$ for which there is an exact sequence
\begin{align}
Q_{n} \xto{f_{n}} \cdots \to Q_{1} \xto{f_{1}} Q_{0} \xto{f_{0}} Y \to 0 \notag
\end{align}
such that $Q_{i} \in \add Q$ for each $0 \le i \le n$.
Let $\Fac_{0} (Q) := \mod A$. 
Moreover, we write $\dim{Q}X = n$ if $Q_{i} \neq 0$, the natural epimorphism $Q_{i} \to \im f_{i}$ is right minimal for each $0 \le i \le n$, and $f_{n}$ can be chosen to be a monomorphism.
\end{itemize}
\end{definition}

Note that if $m,n$ are integers with $m\ge n$, then $\Sub^{m}(Q) \subseteq \Sub^{n}(Q)$ holds.

We collect basic results for approximations.

\begin{lemma} \label{lem-approximation}
Let $Q$ be an $A$-module and let 
\begin{align}\label{ses-lem22}
0 \to X \xto{f} Z \xto{g} Y \to 0 
\end{align}
be a non-split exact sequence with $Z\in \add Q$.
Then the following statements hold. 
\begin{itemize}
\item[(1)] Assume that $f$ is a left $\add Q$-approximation and let $f': X' \to Z'$ be a minimal left $\add Q$-approximation of $X' \in \add X$. 
Then $f'$ is a monomorphism, $Z' \in \add Z$ and $\cok f' \in \add Y$.
\end{itemize}
In the following, we assume that $\Ext_{A}^{1}(Q, Q)=0$.
\begin{itemize}
\item[(2)] 
$f$ is a left $\add Q$-approximation of $X$ if and only if $\Ext_{A}^{1}(Y, Q)=0$. 
\item[(3)] The following conditions are equivalent.
\begin{itemize}
\item[(a)] $X$ is indecomposable, $f$ is a minimal left $\add Q$-approximation of $X$ and $\Ext_{A}^{1}(Q,X)$ vanishes.
\item[(b)] $Y$ is indecomposable, $g$ is a minimal right $\add Q$-approximation of $Y$ and $\Ext_{A}^{1}(Y,Q)$ vanishes.
\end{itemize}
\end{itemize}
\end{lemma}

For the convenience of the readers, we give a proof, although it is a well-known result.

\begin{proof}
(1) For simplicity, let $X' \in \add X$ be an indecomposable module.
Let $\mu: X'\to X$ and $\pi: X\to X'$ be morphisms satisfying $\pi\mu=\mathrm{id}_{X'}$.
Since $f'$ and $f$ are left $\add Q$-approximations, there exist $\mu': Z'\to Z$ and $\pi': Z \to Z'$ such that $\mu'f'=f\mu$ and $\pi'f=f'\pi$ respectively.
It follows from $\mu'f'=f\mu$ that $f'$ is a monomorphism.
By the universal property of cokernels, there exist $\mu'':\cok f' \to Y$ and $\pi'': Y\to \cok f'$.
Therefore we obtain the following commutative diagram:
\begin{align}
\xymatrix{
0\ar[r]&X'\ar[r]^{f'}\ar@{=}[d]&Z'\ar[r]\ar[d]^{\pi'\mu'}&\cok f'\ar[r]\ar[d]^{\pi''\mu''}&0\; \\
0\ar[r]&X'\ar[r]^{f'}&Z'\ar[r]&\cok f'\ar[r]&0.
}\notag
\end{align}
Since $f'$ is left minimal, $\pi'\mu'$ and $\pi''\mu''$ are isomorphisms. 
Hence we have the assertion.

(2) Applying $\Hom_{A}(-, Q)$ to \eqref{ses-lem22} gives an exact sequence 
\begin{align}
\Hom_{A}(Z, Q) \xrightarrow{\Hom(f,Q)} \Hom_{A}(X,Q)  \to \Ext_{A}^{1}(Y,Q) \to \Ext_{A}^{1}(Z,Q)=0. \notag
\end{align}
Then $\Hom(f,Q)$ is surjective if and only if $\Ext_{A}^{1}(Y,Q)=0$.
Hence the assertion follows from the definition of a left $\add Q$-approximation.

(3) We only prove (a)$\Rightarrow$(b); the proof of (b)$\Rightarrow$(a) is similar.
First, it follows from (2) and its dual statement that $\Ext_{A}^{1}(Y, Q)=0$ and $g$ is a right $\add Q$-approximation respectively. 

Next we prove that $g$ is right minimal.
Let $k \in \End_{A}(Z)$ be a morphism with $gk=g$. 
Since $g(\mathrm{id}_{Z}- k)=0$, there exists $l: Z \to X$ such that $fl=\mathrm{id}_{Z}-k$. 
On the other hand, we have $gkf=gf=0$. 
There exists $h \in \End_{A}(X)$ such that $fh=kf$.
Hence we obtain
\begin{align}
flf=(\mathrm{id}_{Z} -k)f=f-kf=f-fh=f(\mathrm{id}_{X}-h). \notag 
\end{align}
Since $f$ is a monomorphism, $lf=\mathrm{id}_{X}-h$.
Thus $\mathrm{id}_{X}-h$ is a non-isomorphism.
Since $\End_{A}(X)$ is local, $h$ is an isomorphism, and hence so is $k$.

Finally, we show that $Y$ is indecomposable. 
Let $\varphi_{1}, \varphi_{2} \in \End_{A}(Y)$ be non-isomorphisms. 
Since $g$ is a right $\add Q$-approximation, we have the following commutative diagram for each $i \in \{1, 2\}$. 
\begin{align}
\xymatrix{
0\ar[r]&X\ar[r]^{f}\ar[d]^{\varphi''_{i}}&Z\ar[r]^{g}\ar[d]^{\varphi'_{i}}&Y\ar[r]\ar[d]^{\varphi_{i}}&0\; \\
0\ar[r]&X\ar[r]^{f}&Z\ar[r]^{g}&Y\ar[r]&0.
}\notag
\end{align}
To show the claim, we have only to prove that $\varphi_{1}+\varphi_{2}$ is not an isomorphism.
Suppose to the contrary that $\varphi_{1} + \varphi_{2}$ is an isomorphism. 
Since $g$ is a right $\add Q$-approximation, there exists $\psi \in \End_{A}(Z)$ such that $(\varphi_{1}+\varphi_{2})^{-1}g=g\psi$.
Then
\begin{align}
g\psi(\varphi'_{1}+\varphi'_{2})=(\varphi_{1}+\varphi_{2})^{-1}g(\varphi'_{1}+\varphi'_{2})=(\varphi_{1}+\varphi_{2})^{-1}(\varphi_{1}+\varphi_{2})g=g. \notag
\end{align}
By minimality of $g$, the morphism $\psi(\varphi'_{1}+ \varphi'_{2})$ is an isomorphism. 
This implies that $\varphi'_{1}+\varphi'_{2}$ is an isomorphism, and hence so is $\varphi''_{1}+\varphi''_{2}$.
Since $\End_{A}(X)$ is local,  at least one of $\varphi''_{1}$ and $\varphi''_{2}$ is an isomorphism.
Without loss of generality, we may assume that $\varphi''_{1}$ is an isomorphism.
Since $f$ is a left $\add Q$-approximation, there exists $\psi_{1} \in \End_{A}(Z)$ such that $\psi_{1} f= f \varphi''^{-1}_{1}$. 
Thus we have
\begin{align}
\psi_{1} \varphi'_{1} f = \psi_{1}  f  \varphi''_{1} = f  \varphi''^{-1}_{1}  \varphi''_{1} = f. \notag
\end{align}
By minimality of $f$, the morphism $\psi_{1} \varphi'_{1}$ is an isomorphism. 
Thus $\varphi'_{1}$ is an isomorphism.
Hence  $\varphi_{1}$ is also an isomorphism, a contradiction. 
Therefore $Y$ is indecomposable.
\end{proof}

Let $Q$ be an $A$-module.
In the following, we study a property of an exact sequence associated with an object in $\Sub^{n+1}(Q)$. Let $M\in \Sub^{n+1}(Q)$, that is, there exists an exact sequence
\begin{align}\label{lex-sec2}
0\to M\xto{f^{0}} Q^{0}\xto{f^{1}} Q^{1}\to \cdots \xto{f^{n}}Q^{n}
\end{align}
such that $Q^{i}\in\add Q$ for each $0\le i\le n$. 
Let $M^{0}:=M$ and $M^{i}=\cok f^{i-1}$ for each $1\le i\le n+1$.

We give a sufficient condition for $\Sub^{n+1}(Q)$ to be closed under direct summands. 

\begin{lemma}\label{lem-lex}
Keep the notation above.
Assume that the inclusion $\iota^{i}:M^{i}\to Q^{i}$ is a left $\add Q$-approximation for each $0\le i\le n$.
For each $X\in \add M$, there exists an exact sequence 
\begin{align}
0 \to X \xto{f_{X}^{0}} Q_{X}^{0} \xto{f_{X}^{1}} Q_{X}^{1} \to \cdots \xto{f_{X}^{n}} Q_{X}^{n}\notag
\end{align}
such that $Q_{X}^{i}\in \add Q^{i}$ and $\cok f_{X}^{i} \in \add M^{i+1}$ for each $0\le i\le n$. 
In particular, $\Sub^{n+1}(Q)$ is closed under direct summands.
\end{lemma}
\begin{proof}
This follows from repeated use of Lemma \ref{lem-approximation}(1).
\end{proof}

By Lemmas \ref{lem-approximation}(2) and \ref{lem-lex}, we have the following proposition.

\begin{proposition}\label{prop-clds}
Let $Q$ be an $A$-module with $\Ext_{A}^{1}(Q,Q)=0$.
Keep the notation concerning the exact sequence \eqref{lex-sec2}.
Then the following statements hold.
\begin{itemize} 
\item[(1)] If $\id Q\le 1$, then $\Ext_{A}^{1}(M^{i+1},Q)=0$ for each $0\le i\le n$.
In particular, $\Sub^{n+1}(Q)$ is closed under direct summands. 
\item[(2)] If $\pd Q\le 1$, then $\Ext_{A}^{1}(Q,M^{i+1})=0$ for each $0\le i\le n$.
In particular, $\Fac_{n+1}(Q)$ is closed under direct summands.
\end{itemize}
\end{proposition}
\begin{proof}
We only show (1); the proof of (2) is similar.
Let $M \in \Sub^{n+1}(Q)$, that is, there exists an exact sequence \eqref{lex-sec2}.
Applying $\Hom_{A}(-, Q)$ to the exact sequence $0 \to M^{i+1} \xto{\iota^{i+1}} Q^{i+1} \to M^{i+2} \to 0$ gives an exact sequence
\begin{align}
\Ext_{A}^{1}(Q^{i+1}, Q) \to \Ext_{A}^{1}(M^{i+1}, Q) \to \Ext_{A}^{2}(M^{i+2}, Q). \notag
\end{align}
Since the left-hand side and right-hand side vanish by the assumption on $Q$, we have $\Ext_{A}^{1}(M^{i+1},Q)=0$.
Moreover, by Lemmas \ref{lem-approximation}(2) and \ref{lem-lex}, $\Sub^{n+1}(Q)$ is closed under direct summands.
\end{proof}

Throughout this paper, we frequently assume that $Q$ is an injective $A$-module with $\pd Q\le 1$. 
By Proposition \ref{prop-clds}, the subcategories $\Sub^{n+1}(Q)$ and $\Fac_{n+1}(Q)$ are closed under direct summands for each $n\ge 0$.
Therefore a direct summand $X$ of $M\in \Sub^{n+1}(Q)$ is also in $\Sub^{n+1}(Q)$. 
Under a certain condition, we determine an exact sequence associated with $X$ as follows.

\begin{proposition}\label{prop-mutation-resol}
Let $Q$ be an $A$-module with $\pd Q \le 1$ and $\Ext_{A}^{1}(Q,Q)=0$.
Keep the notation concerning the exact sequence \eqref{lex-sec2}.
Assume that $Q^{j}\neq 0$, the inclusion $\iota^{j}:M^{j}\to Q^{j}$ is a left $\add Q$-approximation for $0\le j\le n$, and $f^{n}$ is not an epimorphism.
Fix an integer $1\le d\le n$.
Let $X\in \add M^{d}$ be an $A$-module with $X \notin \add Q$.
Then for each $d\le i \le n$, there exists an exact sequence
\begin{align}
0\to X \xto{f_{X}^{d}} Q^{d}_{X} \xto{f^{d+1}_{X}} Q^{d+1}_{X} \to \cdots \xto{f^{i}_{X}} Q^{i}_{X}\to M_{X}^{i+1}\to 0 \notag
\end{align}
such that $0 \neq Q^{k}_{X}\in \add Q^{k}$ for $d\le k\le i$, $M_{X}^{i+1} \in \add M^{i+1}$ and $\add M^{i+1}_{X}\cap \add Q=\{ 0\}$.
Moreover, we have the following statements.
\begin{itemize}
\item[(1)] Let $X$ be an indecomposable module $($satisfying $X\notin\add Q$ if $d=1$$)$.
Then $M^{i+1}_{X}$ is indecomposable.
\item[(2)] Let $X,X'\in \add M^{d}$ be $A$-modules $($satisfying $\add X \cap \add Q= \add X' \cap \add Q=\{0\}$ if $d=1$$)$. 
Then $X\cong X'$ if and only if $M^{i+1}_{X} \cong M^{i+1}_{X'}$.
\item[(3)] For each $A$-module $Y\in \add M^{i+1}$, there exists a unique $A$-module $X\in \add M^{d}$ such that $Y\cong M^{i+1}_{X}$.
\end{itemize}
In particular, $\add M^{i+1} \cap \add Q=\{ 0\}$ for each $1\le i\le n$.
\end{proposition}
\begin{proof}
For simplicity, we assume that $X \in \add M^{d}$ is an indecomposable module which is not contained in $\add Q$. 

(i) By Lemma \ref{lem-approximation}(1), there exists an exact sequence 
\begin{align}
0\to X\to Q_{X}^{d}\to M_{X}^{d+1}\to 0 \notag
\end{align}
such that $0\neq Q_{X}^{d}\in \add Q^{d}$ and $M_{X}^{d+1}\in \add M^{d+1}$. 
Since $X\notin \add Q$, the exact sequence is non-split.
We show that $M_{X}^{d+1}$ is an indecomposable module which is not contained in $\add Q$.
By Proposition \ref{prop-clds}(2), $\Ext_{A}^{1}(Q,M^{d})=0$, and hence $\Ext_{A}^{1}(Q,X)=0$.
Due to Lemma \ref{lem-approximation}(3), $M_{X}^{d+1}$ is indecomposable and $\add M_{X}^{d+1}\cap \add Q=\{ 0\}$.

(ii) By repeated use of (i), there exists an exact sequence
\begin{align}
0\to X \xto{f_{X}^{d}} Q^{d}_{X} \xto{f^{d+1}_{X}} Q^{d+1}_{X} \to \cdots \xto{f^{i}_{X}} Q^{i}_{X}\to M^{i+1}_{X}\to 0 \notag
\end{align}
such that $0\neq Q^{k}_{X}\in \add Q^{k}$ and $M_{X}^{k+1}:=\cok f_{X}^{k} \in \add M^{k+1}$ for each $d\le k \le i$.
Moreover, we have $M_{X}^{k+1}$ is indecomposable and $\add M_{X}^{k+1}\cap \add Q=\{ 0\}$.

By (ii), $M_{X}^{i+1}$ is indecomposable, and hence (1) is clear.
Moreover, (2) follows from uniqueness of a minimal left/right $\add Q$-approximation.
Finally we show (3). 
By the construction of the exact sequence, we obtain $M^{i+1}=\bigoplus_{X}M_{X}^{i+1}$, where $X$ runs over all indecomposable direct summands of $M^{d}$ which are not contained in $\add Q$.
This finishes the proof.
\end{proof}

Now we introduce the following central notion of this paper.

\begin{definition}
Fix an integer $n\ge 0$. Let $I$ be an injective $A$-module and $X$ an $A$-module.
Then we write $\gdom{I}X\ge n$ if $X\in \Sub^{n}(I)$. 
In this case, we say that the \emph{ dominant dimension of $X$ with respect to $I$} is at least $n$.
\end{definition}

For $i \ge 0$, fix a basic right (respectively, left) injective $A$-module $I_{i}$ (respectively, $J_{i}$) with the property that $\add I_{i}$ (respectively, $\add J_{i}$) consists of all injective right (respectively, left) $A$-modules with projective dimension at most $i$.
If $i=0$, then we have $\gdom{I_{0}} X=\dom X$ for each $X \in \mod A$.
Moreover, the $(l, n)$-condition in \cite{I105} and \cite{I205} becomes, in our terminology, the condition that $\gdom{I_{l-1}} A\ge n$.

\begin{remark}\label{rem-I211}
The notion of dominant dimension with respect to injective modules is not always left-right symmetric. 
Namely, there exists an example of an artin algebra $A$ satisfying $\gdom{I_{1}}A \neq \gdom{J_{1}}A^{\op}$ (see, \cite[Remark 2.1.1(2)]{I105}). 
On the other hand, we have $\dom A = \dom A^{\op}$ by \cite[Theorem 4]{M68}.
\end{remark}

\section{Tilting modules and dominant dimension}
In this section, we study a relationship between tilting modules with finite projective dimension and dominant dimension with respect to injective modules.
We start this section with recalling the definition and basic properties of tilting modules.

\begin{definition}(\cite{BB80, HR82, M86})\label{def-tilting}
Fix an integer $d \geq 0$ and let $T,C$ be $A$-modules.
\begin{itemize}
\item[(1)] We call $T$ a \emph{tilting module} if it satisfies the following conditions:
\begin{itemize}
\item[(T1)] $\pd T <\infty$;
\item[(T2)] $\Ext_{A}^{i} (T,T)=0$ holds for all $i \geq 1$; 
\item[(T3)] $\codim{T}A<\infty$.
\end{itemize}
Moreover, a tilting module $T$ is called a \emph{$d$-tilting module} if $\pd T =d$. 
\item[(2)] We call $C$ a \emph{cotilting module} if it satisfies the following conditions:
\begin{itemize}
\item[(C1)] $\id C <\infty$;
\item[(C2)] $\Ext_{A}^{i} (C, C)=0$ holds for all $i \geq 1$; 
\item[(C3)] $\dim{C}\kD A<\infty$.
\end{itemize}
Moreover, a cotilting module $C$ is called a \emph{$d$-cotilting module} if $\id C =d$.
\end{itemize}
\end{definition}

Note that $T$ is a tilting $A$-module if and only if $\kD T$ is a cotilting $A^{\op}$-module.
We collect well-known results for tilting modules with finite projective dimension.
We denote by $\tilt A$ the set of isomorphism classes of basic tilting $A$-modules.
For $M,M' \in \mod A$, we write $M \succeq M'$ if $\Ext_{A}^{i}(M,M')=0$ for all $i \ge 1$.
We denote by $|M|$ the number of isomorphism classes of indecomposable direct summands of $M$.

\begin{proposition}\label{prop-basic-result-tilting}
The following statements hold.
\begin{itemize}
\item[(1)] $($\cite[Theorem 1.4]{M86} \textnormal{and} \cite[Lemma III.2.2]{H88}$)$ Let $T$ be an $A$-module with $\pd T< \infty$ and $\Ext_{A}^{i}(T, T)=0$ for all $i \ge 1$. 
Then $T$ is a tilting module if and only if $\codim{T}A = \pd T$.
\item[(2)] $($\cite[Corollary to Proposition 1.18]{M86}$)$ Let $T$ be a $d$-tilting module and let
\begin{align}
0\to A \to T^{0} \to T^{1} \to \cdots \to T^{d}\to 0 \notag
\end{align}
be an exact sequence in $\mathrm{(T3)}$.
Then $\add T =\add(T^{0}\oplus T^{1}\oplus \cdots \oplus T^{d})$.
\item[(3)] $($\cite[Theorem 1.19]{M86}$)$ If $T$ is a tilting $A$-module, then we have $|T|=|A|$.
\item[(4)] $($\cite{RS91,HU051}$)$ $\succeq$ gives a partial order on $\tilt A$. 
Moreover, if $T \succeq T'$ in $\tilt A$, then $\pd T \le \pd T'$ holds.
\end{itemize}
\end{proposition}

We recall the notion of (left) mutations of tilting modules with finite projective dimension (see \cite{RS91, HU052, CHU94} for details).
Let $T=X\oplus U$ be an $A$-module with $X$ non-zero.
Take a minimal left $\add U$-approximation $f:X\to \overline{U}$ of $X$.
We call $\mu_{X}(T):=\cok f \oplus U$ a \emph{mutation} of $T$ with respect to $X$.
Assume that $X\in \Sub^{1}(U)$. 
If $T$ is a tilting $A$-module, then $T':=\mu_{X}(T)$ is also a tilting $A$-module.
Moreover, we have $T \succeq T'$.

Next, from an exact sequence associated with $A \in \Sub^{n+1}(Q)$, we construct a $d$-tilting module which is contained in $\Fac_{d}(Q)$.
Conversely, by using a $d$-tilting $A$-module $T\in \Fac_{d}(Q)$, we give a construction of an exact sequence which induces $A\in \Sub^{d}(Q)$.
We need the following useful lemmas.

\begin{lemma}[{see \cite[Proposition A.4.7]{ASS06}}]\label{lem-ASS-A47}
Let $0 \to X \to Y \to Z \to 0$ be an exact sequence in $\mod A$. 
Then we have $\pd Z \leq {\rm max} \{\pd X +1, \pd Y \}$ and the equality holds if $\pd X \neq \pd Y$.
\end{lemma}

\begin{lemma}[{see \cite[Lemma 1.1]{M86}}]\label{lem-miy}
Let
\begin{align}
0\to X^{0} \to Y^{0} \xto{f^{1}} Y^{1} \to \cdots \xto{f^{n-2}} Y^{n-2}\xto{f^{n-1}} Y^{n-1} \to X^{n} \to 0 \notag
\end{align}
be an exact sequence and $X^{i}:=\im f^{i}$ for each $1\le i \le n-1$.
Then the following statements hold for an $A$-module $Q$ and an integer $j\ge 1$.
\begin{itemize}
\item[(1)] If both $\Ext_{A}^{j+k}(Y^{k},Q)$ and $\Ext_{A}^{j+k+1}(Y^{k},Q)$ vanish for all $0\le k\le n-1$, then we have
\begin{align}
\Ext_{A}^{j}(X^{0},Q) \cong \Ext_{A}^{j+1}(X^{1},Q) \cong \cdots \cong \Ext_{A}^{j+n}(X^{n},Q). \notag
\end{align}
\item[(2)] If both $\Ext_{A}^{j+n-1-k}(Q,Y^{k})$ and $\Ext_{A}^{j+n-k}(Q,Y^{k})$ vanish for all $0\le k\le n-1$, then we have
\begin{align}
\Ext_{A}^{j}(Q,X^{n}) \cong \Ext_{A}^{j+1}(Q,X^{n-1}) \cong \cdots \cong \Ext_{A}^{j+n}(Q,X^{0}). \notag
\end{align}
\end{itemize}
\end{lemma}

The following result plays an important role in this paper.

\begin{proposition}\label{prop-construct-tilting}
Let $Q$ be an $A$-module with $\pd Q \le 1$ and $\Ext_{A}^{1}(Q,Q)=0$. 
Fix an integer $n\ge 0$.
Then the following statements hold.
\begin{itemize}
\item[(1)] Let
\begin{align}\label{addQ-resolution}
0 \to A \xto{f^{0}} Q^{0} \xto{f^{1}} Q^{1} \to \cdots \xto{f^{n}} Q^{n}
\end{align}
be an exact sequence with $0\neq Q^{i}\in\add Q$ for each $0\le i\le n$. Let $A^{0}:=A$ and $A^{d}:=\cok f^{d-1}$ for $1\le d \le n+1$. 
Assume that the inclusion $\iota^{i}:A^{i}\to Q^{i}$ is a minimal left $\add Q$-approximation for each $0\le i \le n$ and $f^{n}$ is not an epimorphism. 
Then for each $1 \le d \le n+1$, the module $\mathbf{T}^{d}:= A^{d}\oplus Q$ is a $d$-tilting module which is contained in $\Fac_{d}(Q)$.
In particular, $\mathbf{T}^{d+1}$ is a mutation of $\mathbf{T}^{d}$ with respect to $A^{d}$.
\item[(2)] Assume that $T\in\Fac_{d}(Q)$ is a $d$-tilting $A$-module with $Q\in \add T$ for some integer $1\le d \le n+1$.
Then there exists an exact sequence
\begin{align*}
0 \to A \to Q^{0} \to Q^{1} \to \cdots \to Q^{d-1} \to T^{d} \to 0
\end{align*}
such that $0\neq Q^{i} \in \add Q$ for each $0 \le i \le d-1$ and $0 \neq T^{d} \in \add T$. 
In particular, $A \in \Sub^{d}(Q)$.
\end{itemize}
\end{proposition}
\begin{proof}
(1) We check that $T:=\mathbf{T}^{d}$ satisfies the conditions (T1), (T2), and (T3) in Definition \ref{def-tilting}.

(T1) We prove $\pd T\le d$.
This follows from repeated use of Lemma \ref{lem-ASS-A47}.

(T2) We prove $\Ext_{A}^{i}(T,T)=0$ for each $1\le i\le d$. Clearly, we obtain 
\begin{align}
\Ext_{A}^{i}(T,T)
&\cong \Ext_{A}^{i}(Q,Q)\oplus \Ext_{A}^{i}(Q,A^{d})\oplus \Ext_{A}^{i}(A^{d},Q)\oplus \Ext_{A}^{i}(A^{d},A^{d})\notag\\
&\cong \Ext_{A}^{i}(Q,A^{d})\oplus \Ext_{A}^{i}(A^{d},Q)\oplus \Ext_{A}^{i}(A^{d},A^{d}). \notag
\end{align}
First we show $\Ext_{A}^{i}(Q,A^{d})=0$. 
Applying $\Hom_{A}(Q,-)$ to the exact sequence $0\to A^{d-1}\to Q^{d-1}\to A^{d}\to 0$ induces an exact sequence
\begin{align}
\Ext_{A}^{i}(Q,Q^{d-1})\to \Ext_{A}^{i}(Q,A^{d})\to \Ext_{A}^{i+1}(Q,A^{d-1}). \notag
\end{align}
By the assumption on $Q$, we have $\Ext_{A}^{i}(Q,Q^{d-1})=0$ and $\Ext_{A}^{i+1}(Q,A^{d-1})=0$. 
Hence $\Ext_{A}^{i}(Q,A^{d})=0$.
Secondly, we show $\Ext_{A}^{i}(A^{d},Q)=0$.
By the assumption on $Q$, $\Ext_{A}^{j}(Q,Q)=0$ for all $j\ge 1$. Thus it follows from Lemma \ref{lem-miy}(1) that
\begin{align}
\Ext_{A}^{1}(A^{d-i+1},Q)\cong \Ext_{A}^{2}(A^{d-i+2},Q)\cong \cdots \cong \Ext_{A}^{i}(A^{d},Q). \notag
\end{align}
Since $\iota^{d-i}$ is a left $\add Q$-approximation, we have $\Ext_{A}^{1}(A^{d-i+1},Q)=0$ by Lemma \ref{lem-approximation}(2).
In particular, $\Ext_{A}^{j}(A^{d},Q)=0$ for all $j\ge 1$.
Finally, we show $\Ext_{A}^{i}(A^{d},A^{d})=0$.
Since $\Ext_{A}^{j}(A^{d},Q)=0$ for each $j\ge 1$, we obtain isomorphisms
\begin{align}
\Ext_{A}^{i}(A^{d},A^{d})\cong \Ext_{A}^{i+1}(A^{d},A^{d-1})\cong \cdots \cong \Ext_{A}^{i+d}(A^{d},A^{0})\notag
\end{align}
by Lemma \ref{lem-miy}(2).
Since $\pd A^{d}\le d$, we have $\Ext_{A}^{i+d}(A^{d},A^{0})=0$. 
Hence $\Ext_{A}^{i}(A^{d},A^{d})=0$.

(T3) We prove $\codim{T}A=d$. By \eqref{addQ-resolution}, we have the exact sequence
\begin{align}
0 \to A \xto{f^{0}} Q^{0} \xto{f^{1}} Q^{1} \to \cdots \xto{f^{d-1}} Q^{d-1}\to A^{d}\to 0.\notag
\end{align}
This sequence implies the desired property.

(2) Since $\codim{T}A=\pd T =d$ holds by Proposition \ref{prop-basic-result-tilting}(1), there exists an exact sequence
\begin{align}\label{Tdom-resolution}
0 \to A \xto{f^{0}} T^{0} \xto{f^1} T^{1} \to \cdots \xto{f^{d-1}} T^{d-1} \xto{f^{d}} T^{d} \to 0 
\end{align}
with $0\neq T^{i} \in \add T$. 
We put $A^{0} := A$ and $A^{i}:=\cok f^{i-1}$ for $1 \leq i \leq d$.
Without loss of generality, we can assume that the inclusion $\iota^{i}:A^{i}\to T^{i}$ is a minimal left $\add T$-approximation by Lemmas \ref{lem-approximation}(2) and \ref{lem-miy}(1).

In the following, we claim $T^{i} \in \add Q$ for all $0 \leq i \leq d-1$.
By Proposition \ref{prop-clds}(2), we have $T^{i}\in\Fac_{d}(Q)$.
Thus there exists an exact sequence
\begin{align}
Q_{d-1}^{i}\xto{g_{d-1}^{i}} \cdots \to Q_{i}^{i} \xto{g_{i}^{i}}\cdots \to Q_{1}^{i} \xto{g_{1}^{i}} Q_{0}^{i} \xto{g_{0}^{i}}T^{i} \to 0. \notag
\end{align}
such that $Q^{i}_{0},Q^{i}_{1},\ldots,Q^{i}_{d-1}\in \add Q$.
Let $T_{j}^{i}:=\ker g_{j-1}^{i}$ for each $1\le j\le d$.
Applying $\Hom_{A}(A^{i},-)$ to the exact sequence $0\to T_{1}^{i}\to Q_{0}^{i}\to T^{i}\to 0$, we obtain an exact sequence
\begin{align}
\Hom_{A}(A^{i},Q_{0}^{i}) \to \Hom_{A}(A^{i},T^{i})\to \Ext_{A}^{1}(A^{i},T_{1}^{i}). \notag
\end{align}
It is enough for us to show that $\Ext_{A}^{1}(A^{i},T_{1}^{i})=0$.
Indeed, if it is true, then there exists $h^{i}:A^{i}\to Q_{0}^{i}$ such that $\iota^{i}=g_{0}^{i}h^{i}$.
Since $\iota^{i}$ is a left $\add T$-approximation and $Q_{0}^{i} \in \add T$, we have a morphism $\alpha:T^{i}\to Q_{0}^{i}$ with $h^{i}=\alpha \iota^{i}$. 
Therefore we obtain $\iota^{i}=g^{i}_{0}h^{i}=g^{i}_{0}\alpha\iota^{i}$.
By minimality of $\iota^{i}$, the morphism $g^{i}\alpha$ is an isomorphism.
Hence $T^{i}\in\add Q$.

In the rest, we prove $\Ext_{A}^{1}(A^{i},T_{1}^{i})=0$ for each $0 \le i\le d-1$. First we show $\Ext_{A}^{j}(A^{i},Q)=0$ for all $j\ge 1$ and $0\le i \le d$.
Since $\Ext_{A}^{k}(T,Q)=0$ for each $k\ge 1$, we have isomorphisms
\begin{align}
\Ext_{A}^{j}(A^{i},Q)\cong \Ext_{A}^{j+1}(A^{i+1},Q)\cong \cdots \cong \Ext_{A}^{j+d-i}(A^{d},Q)\cong\Ext_{A}^{j+d-i}(T^{d},Q)=0\notag
\end{align}
by Lemma \ref{lem-miy}(1). 
Next, we show $\Ext_{A}^{1}(A^{i},T_{1}^{i})=0$ by induction on $i$. 
If $i=0$, then this is clear. 
Assume $i \ge 1$.
Since $\Ext_{A}^{j}(A^{i}, Q)=0$ holds for all $j\ge 1$, it follows from Lemma \ref{lem-miy}(2) that
\begin{align}
\Ext_{A}^{1}(A^{i},T_{1}^{i})\cong \Ext_{A}^{2}(A^{i},T_{2}^{i}) \cong \cdots \cong \Ext_{A}^{i+1}(A^{i},T_{i+1}^{i}). \notag
\end{align}
By induction hypothesis, $\Ext_{A}^{1}(A^{j},T_{1}^{j})=0$ holds for each $0\le j \le i-1$. Therefore $T^{j} \in \add Q$.
Since $\pd Q\le 1$, we obtain $\pd A^{i}\le i$ by repeated use of Lemma \ref{lem-ASS-A47}. 
This implies $\Ext_{A}^{i+1}(A^{i},T_{i+1}^{i})=0$, and hence $\Ext_{A}^{1}(A^{i},T_{1}^{i})=0$.
Therefore $T^{i} \in \add Q$ for each $0 \le i \le d-1$.
By letting $Q^{i}:=T^{i}$ for each $0\le i\le d-1$, \eqref{Tdom-resolution} is the desired sequence.
\end{proof}

In Proposition \ref{prop-construct-tilting}(1), let $Q$ be a maximal projective-injective direct summand of $A$. 
Then $\mathbf{T}^{d}=A^{d}\oplus Q$ coincides with the tilting module which is described in \cite{CX16, CBS17, NRTZ19, PrSa}. 
Thus Proposition \ref{prop-construct-tilting}(1) can be regarded as one of generalizations of their results.

In the following, we give an example of Proposition \ref{prop-construct-tilting}(1).
Let $\mathbb{T}^{d}$ be a basic module of $\mathbf{T}^{d}$.

\begin{example}
Let $A$ be the algebra defined by the quiver
\begin{align} 
\xymatrix@=15pt{1 \ar[r]^{\alpha}\ar[d]_{\beta} & 2 \ar[d]_{\gamma} \\
 3 \ar[r]_{\delta}  & 4\ar[r]_{\epsilon} & 5 \ar[lu]_{\varphi} 
 }\notag
\end{align}
with relations $\alpha \gamma - \beta \delta$, $\epsilon \varphi$ and $\varphi \gamma$.
Let $Q:= P(1) \oplus X \oplus P(1)/P(3) \oplus P(5)$, where $X:= \cok (P(2) \to P(1) \oplus P(5))$. 
Then $Q$ is a non-injective module with $\pd Q \le 1$ and $\Ext_A^1(Q, Q) =0$. 
We can check that $A$ has an exact sequence
\begin{align}
0 \to A \to P(1)^{\oplus 4} \oplus P(5)^{\oplus 2} \to X^{\oplus 2} \oplus (P(1)/P(3))^{\oplus 2}. \notag
\end{align}
Then we obtain that $\mathbb{T}^{1}=P(1)/P(4)\oplus Q$ is a $1$-tilting $A$-module and $\mathbb{T}^{2}=I(2)\oplus Q$ is a $2$-tilting $A$-module.
\end{example}

If we do not assume $\pd Q\le 1$, then Proposition \ref{prop-construct-tilting}(1) does not necessarily hold as the following example shows.

\begin{example}
Let $A$ be the algebra defined by the quiver
\begin{align} 
\xymatrix@=15pt{1 \ar[rd]_{\gamma} &2 \ar[l]_{\alpha} & 3 \ar[l]_{\beta} \\
 & 4\ar[u]_{\delta}}\notag
\end{align}
with relations $\beta \alpha$, $\gamma \delta$ and $\delta \alpha$. 
Then $A$ has a minimal injective coresolution
\begin{align}
0 \to A \to I(2)^{\oplus 2} \oplus I(4)^{\oplus 2} \to I(2) \oplus I(3) \oplus I(4) \to I(1) \oplus I(3) \oplus I(4) \to I(1) \to 0. \notag
\end{align}
Setting $Q=I(2)\oplus I(3) \oplus I(4)$, we have $\pd Q = 2$ and $A\in \Sub^{2}(Q)$.
Then we have $\mathbb{T}^{1}= S(2) \oplus S(4)\oplus Q$.
However, we obtain that $\Ext_{A}^{1}(\mathbb{T}^{1},\mathbb{T}^{1}) \neq 0$ since $\Ext_{A}^{1}(I(2), S(2)) \cong \Hom_{A}(P(2), S(2)) \neq 0$.
Therefore, in this case, we can not obtain tilting modules by the  construction of Proposition \ref{prop-construct-tilting}(1).
\end{example}

As an application of Proposition \ref{prop-construct-tilting}(1), we give the minimum element in
\begin{align}
\tilt^{\succeq Q}_{d}A:=\{ T \in \tilt A \mid \pd T \le d\ \textnormal{and}\ T \succeq Q \}, \notag
\end{align}
which is an analogue of \cite[Theorem 3.4(2)]{IZ19}. 

\begin{corollary}
Keep the notation in Proposition \ref{prop-construct-tilting}(1).
For each $1\le d \le n+1$, the $d$-tilting $A$-module $\mathbb{T}^{d}$ is the minimum element in $\tilt^{\succeq Q}_{d}A$.
\end{corollary}
\begin{proof}
Let $\mathbb{T}^{d}$ be the basic module of the $d$-tilting module $A^{d}\oplus Q$.
We claim that $\Ext_{A}^{i}(T,\mathbb{T}^{d})=0$ for each tilting $A$-module $T\in \tilt^{\succeq Q}_{d}A$ and all integers $i>0$.
It is enough to show that $\Ext_{A}^{i}(T,A^{d})=0$.
This follows from Lemma \ref{lem-miy}(2).
Hence we have the assertion.
\end{proof}

We give a remark on Bongartz completion.
\begin{remark}
Let $Q$ be a basic $A$-module with $\pd Q \le 1$ and $\Ext_{A}^{1}(Q,Q)=0$.
Namely, it is partial tilting. 
Then Bongartz \cite{B81} describes explicitly an $A$-module $X$ such that $T:=X\oplus Q$ is a basic tilting module with projective dimension at most one.
Note that $\mathbb{T}^{1}$ is not always isomorphic to $T$. 
For example, we assume that $A$ is a path algebra of Dynkin type of $\mathbb{A}$ with linear orientation and take $Q$ a maximal projective-injective direct summand of $A$.
Then we have $T=A$ and $\mathbb{T}^{1}=\kD A$.
\end{remark}

Now we state our main result of this paper.

\begin{theorem}\label{main-thm1}
Fix an integer $n\ge 0$.
Let $A$ be an artin algebra and $I$ an injective $A$-module with $\pd I \le 1$.
Then the following statements are equivalent.
\begin{itemize}
\item[(1)] $\gdom{I}A\ge n+1$.
\item[(2)] For each $0\le d\le \min\{\id A, n+1\}$, $\Fac_{d}(I)\cap \Sub^{n+1-d}(I)$ admits a unique basic $d$-tilting $A$-module.
\item[(3)] There exists an integer $0\le d\le \min\{\id A, n+1\}$ such that $\Fac_{d}(I)\cap \Sub^{n+1-d}(I)$ admits a unique basic $d$-tilting $A$-module.
\end{itemize}
\end{theorem}

In the following, we give a proof of the theorem.
To show the uniqueness of the tilting modules in the statements (2) and (3), we need the following lemma.

\begin{lemma}\label{lem-uniquness-tilting}
Fix an integer $d\ge 1$. 
Let $I$ be an injective $A$-module with $\pd I\le d$ and $\{ X_j \}_{j \in J}$ the set of representatives of the isomorphism classes of all  indecomposable $A$-modules in $\Fac_{d}(I)$ such that $\pd X_{j} \le d$.
Assume that $\Fac_{d}(I)$ admits a basic tilting $A$-module $T$ with $\pd T \le d$.
Then the following statements hold. 
\begin{itemize}
\item[(1)] If $J'$ is a finite subset of $J$, then $X_{J'}\oplus T$ is tilting, where $X_{J'}:=\bigoplus_{j \in J'}X_{j}$. 
\item[(2)] $J$ is a finite set. 
\item[(3)] $T$ is isomorphic to $X_{J}$.
In particular, $\Fac_{d}(I)$ has a unique basic tilting module with projective dimension at most $d$ if it exists.
\end{itemize}
\end{lemma}
\begin{proof}
(1) The conditions (T1) and (T3) clearly hold. 
The condition (T2) follows from Lemma \ref{lem-miy}(2).

(2) Suppose that $J'$ is any finite subset of $J$ with $|A|<|X_{J'}|$.
Since $T$ and $X_{J'}\oplus T$ are tilting $A$-modules, we have $X_{J'}\in \add T$, and hence $|X_{J'}|\le |T|$, a contradiction to Proposition \ref{prop-basic-result-tilting}(3).

(3) By definition, we obtain $T\in \add X_{J}$. 
Since $X_{J}\oplus T$ is also tilting by (1), we have the assertion.
\end{proof}

Now we are ready to prove Theorem \ref{main-thm1}.

\begin{proof}[Proof of Theorem \ref{main-thm1}]
Let $I$ be an injective $A$-module with $\pd I \le 1$.

(2)$\Rightarrow$(3): This is clear.

(3)$\Rightarrow$(1): 
By our assumption, there exists a basic $d$-tilting module $T \in \Fac_d(I) \cap \Sub^{n+1-d}(I)$ for some integer  $0 \le d \le \min\{ \id A, n+1 \}$.
If $d=0$, then we obtain $A=T\in \Sub^{n+1}(I)$, and hence $\gdom{I}A\ge n+1$.
Therefore we assume $d\ge 1$.
By Lemma \ref{lem-uniquness-tilting}(3), $I\in \add T$.
Applying Proposition \ref{prop-construct-tilting}(2) to $T$, we have an exact sequence 
\begin{align}
0 \to A \to I^{0}\to \cdots \to I^{d-1} \to T^{d} \to 0. \notag
\end{align}
such that $I^{j} \in \add I$ for each $0 \le j \le d-1$ and $T^{d} \in  \add T$.
By Proposition \ref{prop-clds}(1), we have $T^{d} \in \Sub^{n+1-d}(I)$.
Thus there exists an exact sequence
\begin{align}
0 \to T^{d} \to I^{d} \to \cdots \to I^{n} \notag
\end{align}
with $I^{i}\in \add I$ for all $d\le i\le n$. 
Composing two exact sequences, we have the following exact sequence:
\begin{align}
\xymatrix{
0\ar[r]&A\ar[r]&I^{0}\ar[r]&\cdots \ar[r]&  I^{d-1}\ar[rr] \ar[rd]&&I^{d}\ar[r]&\cdots\ar[r]&I^{n}.\\
&&&&&T^{d}\ar[ru]&&&
}\notag
\end{align}
Hence we obtain the assertion.

(1)$\Rightarrow$(2): 
If $A$ is self-injective, then there is nothing to prove.
We assume that $A$ is not self-injective.
Let $m:=\min\{\id A -1, n\}$.
Since $\id A \ge 1$ and $\gdom{I}A\ge n+1$, there exists a minimal injective coresolution of $A$
\begin{align}\label{seq-minimal-inj-coresol}
0 \to A \xto{f^{0}} I^{0} \xto{f^{1}} I^{1} \to \cdots \xto{f^{m}} I^{m} \to \cdots 
\end{align}
such that $I^{0},I^{1},\ldots,I^{m}\in \add I$ are non-zero and $f^{m}$ is not an epimorphism.
Note that the inclusion $\iota^{i}: \im f^{i}\to I^{i}$ is a minimal left $\add I$-approximation for each $0\le i \le m$.
For $0\le d \le m+1$, put
\begin{align}
\mathbf{T}^{d}=
\begin{cases}
\ A &(d=0),\\
\ \cok f^{d-1}\oplus I &(d\neq 0).
\end{cases}\notag
\end{align}
If $d=0$, then $\mathbf{T}^{0}$ is clearly a $0$-tilting module.
Since a basic $0$-tilting module is unique (up to isomorphism), $\Fac_{0}(I)\cap \Sub^{n+1}(I)$ admits a unique $0$-tilting module.
Assume $d\neq 0$. 
Due to Proposition \ref{prop-construct-tilting}(1), $\mathbf{T}^{d}\in \Fac_{d}(I)$ is a $d$-tilting $A$-module. 
Moreover, by the definition of $m$, we obtain $\mathbf{T}^{d}\in \Sub^{n+1-d}(I)$. Hence it follows from Lemma \ref{lem-uniquness-tilting}(3) that $\Fac_{d}(I)\cap \Sub^{n+1-d}(I)$ admits a unique basic $d$-tilting module.
The proof is complete.
\end{proof}

\begin{notation}\label{not-unique-tilting}
If the equivalent conditions of Theorem \ref{main-thm1} hold, $\Fac_{d}(I)\cap\Sub^{n+1-d}(I)$ admits a unique basic $d$-tilting $A$-module. 
For each $1\le d\le \min\{\id A, n+1\}$, we denote it by $\mathbb{T}^{d}$.
Then $\mathbb{T}^{d}$ is the basic module of $A^{d}\oplus I$, where $A^{d}$ is the cokernel of $f^{d-1}$ in \eqref{seq-minimal-inj-coresol}.
\end{notation}

\section{Tilting modules and almost Auslander--Gorenstein algebras}
In this section, we give characterizations of almost $n$-Auslander--Gorenstein algebras and almost
$n$-Auslander algebras by the existence of the tilting modules in Theorem \ref{main-thm1}.
We start this section with giving the definition of these algebras.

\begin{definition}
Fix an integer $n\ge 0$.
Let $A$ be an artin algebra and $I$ a basic injective $A$-module with the property that $\add I$ consists of all injective $A$-modules with projective dimension at most one.
\begin{itemize}
\item[(1)] We call $A$ a \emph{minimal almost $n$-Auslander--Gorenstein algebra} if it satisfies
\begin{align}
\id A \le n+1 \le \gdom{I} A. \notag
\end{align}
\item[(2)] We call $A$ an \emph{almost $n$-Auslander algebra} if it satisfies
\begin{align}
\gl{A} \le n+1 \le \gdom{I} A. \notag
\end{align}
\end{itemize}
\end{definition}

Throughout this paper, for brevity we omit the word ``minimal'' in minimal almost $n$-Auslander--Gorenstein algebras.
Here are some examples of almost $n$-Auslander--Gorenstein algebras and almost $n$-Auslander algebras. 

\begin{example}
\begin{itemize}
\item[(1)] Recall that $A$ is a \emph{minimal $n$-Auslander--Gorenstein algebra} (respectively, \emph{$n$-Auslander algebra}) if and only if $\id A \le n+1 \le  \dom A$ (respectively, $\gl A\le n+1 \le \dom A$).
Clearly minimal $n$-Auslander--Gorenstein algebras (respectively, $n$-Auslander algebras) are almost $n$-Auslander--Gorenstein algebras (respectively, almost $n$-Auslander algebras).
\item[(2)] Recall that an artin algebra is called a \emph{relative Auslander algebra} if it is the endomorphism algebra of an additive generator of a faithful torsion-free class of another artin algebra.
Iyama \cite[Theorem 2.1]{I105} shows that an artin algebra $A$ is a relative Auslander algebra if and only if both $A$ and $A^{\op}$ are almost $1$-Auslander algebras.
\end{itemize} 
\end{example}

In the rest of this section, we always assume that $I$ is a basic injective $A$-module with the property that $\add I$ consists of all injective $A$-modules with projective dimension at most one.
We give concrete examples of almost $n$-Auslander--Gorenstein algebras and almost $n$-Auslander algebras. 
The following example may be obtained via a general construction we will explain later in Proposition \ref{prop-hmky}.

\begin{example}\label{ex-alag}
Let $n \ge 4$ be an integer and $A$ the algebra defined by the quiver
\begin{align} 
\xymatrix@R=-2mm{ \ar@(ur,ul)_{\beta_1} & \ar@(ur,ul)_{\beta_2} &  & \ar@(ur,ul)_{\beta_{n-1}}&  \ar@(ur,ul)_{\beta_n} \\
1 \ar[r]^{\alpha_1}  & 2 \ar[r]^{\alpha_2\ } & \cdots \ar[r]^{\alpha_{n-2}\ } & n-1 & n \ar[l]_{\ \ \alpha_{n-1}}  
}\notag
\end{align}
with relations $\alpha_{i}\alpha_{i+1}$ ($1\le i \le n-3$), $\beta_{i}\alpha_{i}-\alpha_{i}\beta_{i+1}$ ($1\le i \le n-2$), $\beta_{i}^{2}$ ($1\le i \le n$) and $\beta_{n}\alpha_{n-1}-\alpha_{n-1}\beta_{n-1}$.
Then we obtain $\gl A =\infty$ and $I=I(2)\oplus I(3)\oplus \cdots \oplus I(n)$.
We can check that $A$ has a minimal injective coresolution
\begin{align}
0\to A \to I^{0} \to I^{1} \to I^{2} \to \cdots \to I^{n-2} \to 0, \notag
\end{align}
where $I^{0}:=I(2)\oplus \cdots \oplus I(n-2) \oplus I(n-1)^{\oplus 3}$, $I^{1}:=I(n-2)^{\oplus 2}\oplus I(n)$, $I^{2}:=I(n-3)^{\oplus 2}$, $\ldots$, and $I^{n-2}:=I(1)^{\oplus 2}$.  
Thus $A$ is an almost $(n-3)$-Auslander--Gorenstein algebra which is not an almost $(n-3)$-Auslander algebra.

If $A'$ is the factor algebra $A/I'$, where $I'$ is the two-sided ideal of $A$ generated by $\beta_{i}$ ($1\le i \le n$), then we can easily check that $A'$ is an almost $(n-3)$-Auslander algebra. 
\end{example}

In the following, we give various properties of almost $n$-Auslander--Gorenstein algebras and almost $n$-Auslander algebras.
First note that $A$ being almost $n$-Auslander--Gorenstein does not imply that $A^{\op}$ is (see Remark \ref{rem-I211} and Example \ref{ex-alag}) although $A$ being minimal $n$-Auslander--Gorenstein implies that $A^{\op}$ is (see \cite[Proposition 4.1(a)]{IS18}).
Next we show that almost $n$-Auslander--Gorenstein algebras are Iwanaga--Gorenstein algebras. 
Recall that an artin algebra $A$ is called an \emph{$n$-Iwanaga--Gorenstein algebra} if it satisfies $\id A \le n$ and $\id A^{\op}\le n$.
Note that if both $\id A$ and $\id A^{\op}$ are finite, then $\id A =\id A^{\op}$ (see \cite[Lemma 6.9]{AR91}).

\begin{proposition}\label{prop-property-alag}
Fix an integer $n\ge 0$ and let $A$ be an artin algebra. Then the following statements hold.
\begin{itemize}
\item[(1)] $A$ is an almost $n$-Auslander--Gorenstein algebra if and only if $A$ is an $(n+1)$-Iwanaga--Gorenstein algebra with $\gdom{I}A\ge n+1$.
\item[(2)] The following statements are equivalent.
\begin{itemize}
\item[(a)] $A$ is an almost $0$-Auslander--Gorenstein algebra.
\item[(b)] $A$ is an almost $n$-Auslander--Gorenstein algebra with $\id A \le 1$.
\item[(c)] $A$ satisfies $\id A \le 1$ and $\gdom{I}A=\infty$.
\item[(d)] $A$ is a $1$-Iwanaga--Gorenstein algebra.
\end{itemize}
In particular, an almost $0$-Auslander--Gorenstein algebra is an almost $m$-Auslander--Gorenstein algebra for all integers $m\ge 0$.
\item[(3)] $A$ is an almost $n$-Auslander--Gorenstein algebra with $\id A>1$ if and only if $A$ satisfies $\id A =n+1=\gdom{I}A$.
\end{itemize}
\end{proposition}

\begin{proposition}\label{prop-property-alaus}
Let $A$ be an artin algebra. 
Then the following statements hold.
\begin{itemize}
\item[(1)] For each $n \ge 0$, almost $n$-Auslander algebras coincide with almost $n$-Auslander--Gorenstein algebras with finite global dimension. 
\item[(2)] Almost $0$-Auslander algebras coincide with hereditary algebras. 
\item[(3)] $A$ is an almost $n$-Auslander algebra with $\gl A >1$ if and only if $A$ satisfies $\gl A = n+1 =\gdom{I}A$.
\end{itemize}
\end{proposition}

To show Propositions \ref{prop-property-alag} and \ref{prop-property-alaus}, we need the following lemma.

\begin{lemma}\label{lem-ARS}
Let $A$ be an artin algebra with $\id A =m+1$ and let
\begin{align}
0\to A \to I^{0} \to I^{1}\to \cdots \to I^{m}\to I^{m+1} \to 0 \notag
\end{align}
be a minimal injective coresolution of $A$.
Then the following statements hold.
\begin{itemize}
\item[(1)] If $\gl A <\infty$, then we have $\id A =\gl A$.
\item[(2)] If $I'\in \add I^{m+1}$, then $\pd I'\ge m+1$.
\end{itemize}
In the following, we also assume that $Q$ is an injective $A$-module with projective dimension at most one and $\gdom{Q}A\ge m+1$.
\begin{itemize}
\item[(3)] We have $\kD A=\mathbb{T}^{m+1}$ $($see Notation \ref{not-unique-tilting}$)$.
\item[(4)] $\inj A =\add(I^{0} \oplus I^{1} \oplus \cdots \oplus I^{m+1})$.
\item[(5)] For each indecomposable injective $A$-module $I'$, we have $\pd I'\in \{ 0,1\}$ or $\pd I' = m+1$.
In particular, if $m \ge 1$, then we obtain that $I'\in \add I^{m+1}$ if and only if $\pd I' = m+1$.
\item[(6)] If $m\ge 1$, then $\id A =\gdom{Q}A$ and $\add Q=\add I$, where $I$ is a basic injective $A$-module with the property that $\add I$ consists of all injective $A$-modules with projective dimension at most one.
\end{itemize}
\end{lemma}
\begin{proof}
(1) This is a well-known result (see \cite[Lemma VI.5.5(b)]{ARS95}).

(2) Let $I'\in \add I^{m+1}$ be an indecomposable module with simple socle $S$.
Then we have $\Hom_{A}(S,I^{m+1})\neq 0$.
Suppose to the contrary $\pd I' \le m$. 
Applying $\Hom_{A}(-, A)$ to an exact sequence $0 \to S \to I' \to I'/S \to 0$, we have 
\begin{align}
\Ext_A^{m+1}(I', A) \to \Ext_{A}^{m+1}(S,A) \to \Ext_{A}^{m+2}(I'/S, A).\notag
\end{align}
Since $\id A =m+1$ holds, the right-hand side vanishes. 
On the other hand, the left-hand side vanishes since $\pd I' \le m$.
Thus we have $\Hom_{A}(S, I^{m+1}) \cong \Ext_{A}^{m+1}(S,A) = 0$, a contradiction.

Now we assume $\gdom{Q} A \ge m+1$. If $\id A =0$, then there is nothing to prove. 
Hence we may assume $m\ge 0$.

(3) By Theorem \ref{main-thm1}(1)$\Rightarrow$(2), the tilting module $\mathbb{T}^{m+1}\in\add (I^{m+1}\oplus I)$ is injective, and hence $\mathbb{T}^{m+1}=\kD A$.

(4) By (3), we obtain $\inj A =\add\kD A= \add \mathbb{T}^{m+1} = \add(I^{0}\oplus I^{1}\oplus \cdots \oplus I^{m+1})$, where the last equality follows from Proposition \ref{prop-basic-result-tilting}(2).

(5) First we show $\pd I^{m+1}=m+1$. Indeed, by repeated use of Lemma \ref{lem-ASS-A47}, we have $\pd I^{m+1} \le \max\{ \pd I^{m}, \pd \im f^{m}+1\} \le \pd A +m+1 = m+1$. 
Thus we obtain that $\pd I^{m+1}=m+1$ by (2). 
Let $I'$ be an indecomposable injective $A$-module. 
By (4), we obtain that either $I' \in \add (I^{0}\oplus \cdots \oplus I^{m})$ or $I' \in \add I^{m+1}$. If $I'\in \add (I^{0}\oplus \cdots \oplus I^{m})$, then $\pd I'\le 1$. On the other hand, if $I'\in\add I^{m+1}$, then $\pd I'=m+1$.

(6) Let $m\ge 1$.
Suppose to the contrary $\gdom{Q}A>m+1$. Then we obtain $\pd I^{m+1}\le 1$, a contradiction to (5). 
Therefore $\id A=\gdom{Q}A$. 
Next we show $\add Q=\add I$. 
Since $\gdom{Q}A=m+1$, we have $\add (I^{0} \oplus \cdots \oplus I^{m}) \subseteq\add Q$.
Hence the assertion holds by (2) and (4).
\end{proof}

Now we are ready to prove Propositions \ref{prop-property-alag} and \ref{prop-property-alaus}.

\begin{proof}[Proof of Proposition \ref{prop-property-alag}]
(1) Let $A$ be an almost $n$-Auslander--Gorenstein algebra.
By definition, $\id A \le n+1$. Thus we have only to show $\pd \kD A \le n+1$. 
Indeed, this follows from Lemma \ref{lem-ARS}(4) and (5). 
The converse is clear.

(2) (d)$\Rightarrow$(c)$\Rightarrow$(b)$\Rightarrow$(a) is clear. We show (a)$\Rightarrow$(d).
By Lemma \ref{lem-ARS}(3), we obtain $\pd \kD A\le 1$.
Hence we have the assertion.

(3) We show the ``only if'' part. 
By assumption, $\id A =m+1 \le n+1 \le \gdom{I}A$ for some integer $m\ge 1$.
Hence the assertion follows from Lemma \ref{lem-ARS}(6).
Next we show the ``if'' part. By assumption, $A$ is an almost $n$-Auslander--Gorenstein algebra.
Thus it is enough to show $\id A>1$. Suppose $\id A \le 1$. Then by (2), we have $\gdom{I}A=\infty$, a contradiction. The proof is complete.
\end{proof}

\begin{proof}[Proof of Proposition \ref{prop-property-alaus}]
(1) follows from Lemma \ref{lem-ARS}(1).
(2) and (3) follow from (1) and Proposition \ref{prop-property-alag}.
\end{proof}

The following theorem is a main result of this section, which is a refinement of \cite[Lemma 1.3]{HU96} and a generalization of \cite[Corollary 3.10]{PrSa}.

\begin{theorem}\label{thm-almost-AG}
Let $A$ be an artin algebra with injective dimension at least two and let $I$ be a basic injective $A$-module with the property that $\add I$ consists of all injective $A$-modules with projective dimension at most one.
Then the following statements are equivalent for each integer $n\ge 1$.
\begin{itemize}
\item[(1)] $A$ is an almost $n$-Auslander--Gorenstein algebra.
\item[(2)] For each $0\le d \le \min\{ \id A, n+1\}$, there exists a unique basic $d$-tilting $A$-module in $\Fac_{d}(I) \cap \Sub^{n+1-d}(I)$ which is an $(n+1-d)$-cotilting $A$-module.
\item[(3)] For some $0\le d \le \min\{ \id A, n+1\}$, there exists a unique basic $d$-tilting $A$-module in $\Fac_{d}(I) \cap \Sub^{n+1-d}(I)$ which is an $(n+1-d)$-cotilting $A$-module.
\item[(4)] For each $0\le d\le n+1$, there exists a unique basic $d$-tilting $A$-module in $\Fac_{d}(I) \cap \Sub^{n+1-d}(I)$ which is an $(n+1-d)$-cotilting $A$-module.
\item[(5)] For some $0\le d \le n+1$, there exists a unique basic $d$-tilting $A$-module in $\Fac_{d}(I) \cap \Sub^{n+1-d}(I)$ which is an $(n+1-d)$-cotilting $A$-module.
\end{itemize}
If in addition we assume $\gl A <\infty$, then the following statement is also equivalent.
\begin{itemize}
\item[(6)] $A$ is an almost $n$-Auslander algebra.
\end{itemize}
\end{theorem}

We give an analogue of the theorem for the case $\id A=1$.

\begin{proposition}\label{prop-id1} 
Let $A$ be an artin algebra with $\id A=1$ and let $I$ be as in Theorem \ref{thm-almost-AG}.
Then the following conditions are equivalent for each $n\ge 0$.
\begin{itemize}
\item[(1)] $A$ is an almost $n$-Auslander--Gorenstein algebra.
\item[(2)] There exists a unique basic $1$-tilting $A$-module in $\Fac_{1}(I) \cap \Sub^{n}(I)$ which is a $0$-cotilting $A$-module.
\end{itemize}
\end{proposition}
\begin{proof}
By Theorem \ref{main-thm1}, we obtain that $\gdom{I}A\ge n+1$ if and only if there exists a unique basic $1$-tilting $A$-module $T$ which is contained in $\Fac_{1}(I)\cap \Sub^{n}(I)$.
In this case, since $\id A=1$, it follows from Lemma \ref{lem-ARS}(3) that $T$ is isomorphic to $\kD A$, and hence $0$-cotilting.
This finishes the proof.
\end{proof} 

In the following, we show Theorem \ref{thm-almost-AG}. 
To show Theorem \ref{thm-almost-AG}(1)$\Rightarrow$(4), we need the following lemma.

\begin{lemma}\label{lem-cotilting}
Let $A$ be an almost $n$-Auslander--Gorenstein algebra with injective dimension at least two. 
Fix an integer $1\le d \le n+1$. Let $\mathbb{T}^{d}$ be as in Notation \ref{not-unique-tilting}. 
Then $\mathbb{T}^{d}$ has an injective coresolution
\begin{align}
0\to \mathbb{T}^{d} \to J^{d}\oplus I \to J^{d+1}\to \cdots \to J^{n-1} \to J^{n}\oplus I \to \kD A \to 0 \notag
\end{align}
with $J^{j}\in \add I$ for each $d\le j \le n$.
In particular, $\mathbb{T}^{d}$ is $(n+1-d)$-cotilting.
\end{lemma}
\begin{proof}
If $d=n+1$, then we have $\mathbb{T}^{n+1}=\kD A$ by Lemma \ref{lem-ARS}(3), and hence there is nothing to prove.
Assume $d\le n$. 
Let $X$ be a basic maximal direct summand of $A^{d}$ which contains no non-zero injective module as a direct summand.
By Proposition \ref{prop-mutation-resol}, we have an exact sequence
\begin{align}
0\to X \to I^{d}_{X} \to I^{d+1}_{X} \to \cdots \to I^{n}_{X} \to {A}^{n+1}_{X} \to 0. \notag
\end{align}
Then $A_{X}^{n+1}$ is injective since $\id A =n+1$.
Letting $J^{j}:=I_{X}^{j}$, we obtain the desired injective coresolution because $\mathbb{T}^{d}=X\oplus I$ and $\mathbb{T}^{n+1}=\kD A={A}^{n+1}_{X}\oplus I$ hold.
\end{proof}

Now we are ready to prove Theorem \ref{thm-almost-AG}.

\begin{proof}[Proof of Theorem \ref{thm-almost-AG}]
(4)$\Rightarrow$(2)$\Rightarrow$(3)$\Rightarrow$(5): This follows from $\{ 1,2, \ldots, \min\{\id A, n+1\}\}\subseteq \{ 1,2,\ldots,n+1\}$.

(1)$\Rightarrow$(4): By Proposition \ref{prop-property-alag}(3), we have $\id A =n+1=\gdom{I}A$.
Then $\mathbb{T}^{0}=A$ is an $(n+1)$-cotilting module which is contained in $\Sub^{n+1}(I)$.
Since $\id A =n+1$, it follows from Theorem \ref{main-thm1}(1)$\Rightarrow$(2) that there exists a unique basic $d$-tilting module $\mathbb{T}^{d}$ which is contained in $\Fac_{d}(I) \cap \Sub^{n+1-d}(I)$ for each $1 \le d \le n+1$.
By Lemma \ref{lem-cotilting}, $\mathbb{T}^{d}$ is an $(n+1-d)$-cotilting module.

(5)$\Rightarrow$(1): Due to our assumption, there exists a unique basic $d$-tilting module $T \in \Fac_d(I) \cap \Sub^{n+1-d}(I)$ for some integer $0 \le d \le  n+1 $.
If $d=0$, then $T\cong A$ is an $(n+1)$-cotilting module which is contained in $\Sub^{n+1}(I)$. Hence we have the assertion.
In the following, we assume $d\ge 1$.
By Lemma \ref{lem-uniquness-tilting}(3), $I\in \add T$.
Applying Proposition \ref{prop-construct-tilting}(2) to $T$ gives an exact sequence
\begin{align}\label{lex-471}
0 \to A \to I^{0}\to \cdots \to I^{d-1} \to T^{d} \to 0
\end{align}
with $0\neq I^{j} \in \add I$ for each $0 \le j \le d-1$ and $0\neq T^{d} \in \add T$.
By Proposition \ref{prop-clds}(1), $T^{d}\in \Sub^{n+1-d}(I)$. Thus there exists a minimal injective coresolution
\begin{align}\label{lex-472}
0\to T^{d}\to I^{d}\to I^{d+1}\to \cdots \xto{f^{n}} I^{n}
\end{align}
such that $I^{d},I^{d+1},\ldots, I^{n}\in \add I$.
If $I^{i}=0$ for some $d \le i \le n$, then $\id A \le 1$ by Lemma \ref{lem-ARS}(2).
This contradicts $\id A \ge 2$. 
Hence we obtain that $I^{i} \neq 0$ for each $d \le i \le n$. 
Since $\id T^{d} \le n+1-d$, we have $\cok f^{n} \in \inj A$.
Therefore we have $\id A =n+1=\gdom{I}A$ by \eqref{lex-471} and \eqref{lex-472}.

If $\gl A<\infty$, then (1)$\Leftrightarrow$(6) follows from Proposition \ref{prop-property-alaus}(1).
This finishes the proof.
\end{proof}

As an application, we can recover Pressland--Sauter's result.
\begin{corollary}[{\cite[Theorem 1]{PrSa}}]\label{PrSa-Theorem1}
Let $A$ be a non-self-injective artin algebra with $\dom A =n+1$ and let $Q$ be a maximal projective-injective direct summand of $A$.
Then the following statements are equivalent for each $n\ge 1$.
\begin{itemize}
\item[(1)] $A$ is a minimal $n$-Auslander--Gorenstein algebra.
\item[(2)] For each $0\le d\le n+1$, there exists a unique basic $d$-tilting $A$-module in $\Fac_{d}(Q) \cap \Sub^{n+1-d}(Q)$ which is an $(n+1-d)$-cotilting $A$-module.
\item[(3)] For some $0\le d\le n+1$, there exists a unique basic $d$-tilting $A$-module in $\Fac_{d}(Q) \cap \Sub^{n+1-d}(Q)$ which is an $(n+1-d)$-cotilting $A$-module.
\end{itemize}
\end{corollary}
\begin{proof}
Let $I$ be as in Theorem \ref{thm-almost-AG}.
If at least one of the statements (1), (2) and (3) is satisfied, then we have $2\le \id A \le \gdom{Q}A=\dom A$.
In this case, we obtain $\add Q=\add I$ by Lemma \ref{lem-ARS}(6).
Hence the assertion follows from Theorem \ref{thm-almost-AG}.
\end{proof}

We end this section by giving a generalization of Example \ref{ex-alag}.

\begin{lemma}\label{lem-tensor}
Let $A,B$ be finite dimensional algebras over a field $\mathbf{k}$.
Then the following statements hold.
\begin{itemize}
\item[(1)] For a projective $A$-module $P$ and a projective $B$-module $Q$, $P\otimes_{\mathbf{k}}Q$ is a projective $A\otimes_{\bf{k}}B$-module.
\item[(2)] Assume that $B$ is a self-injective algebra. 
For an injective $A$-module $I$, $I\otimes_{\mathbf{k}}B$ is an injective $A\otimes_{\mathbf{k}}B$-module.
\end{itemize}
\end{lemma}
\begin{proof}
(1) This is a well-known result (see for example \cite[IX.2.3]{CE56}).

(2) Since $B$ is self-injective, we have $I\otimes_{\mathbf{k}}B\cong \kD(\kD I)\otimes_{\mathbf{k}}\kD B \cong \kD(\kD I\otimes_{\mathbf{k}}B)$.
By (1), $\kD I\otimes_{\mathbf{k}}B$ is a projective $(A\otimes_{\mathbf{k}}B)^{\op}$-module. 
Thus we obtain the assertion.
\end{proof}

We construct almost $n$-Auslander--Gorenstein algebras from almost $n$-Auslander algebras by taking tensor product.

\begin{proposition}\label{prop-hmky}
Let $A,B$ be finite dimensional algebras over a field $\mathbf{k}$.
Assume that $B$ is a self-injective algebra.
If $A$ is an almost $n$-Auslander--Gorenstein algebra, then $A\otimes_{\mathbf{k}}B$ is also an almost $n$-Auslander--Gorenstein algebra which is not an almost $n$-Auslander algebra.
\end{proposition}
\begin{proof}
Since $A$ is an almost $n$-Auslander--Gorenstein algebra, we have $\id A=m+1$ for some integer $m \le n$.
Take a minimal injective coresolution
\begin{align}
0 \to A \to I^{0} \to I^{1} \to \cdots \to I^{m} \to I^{m+1} \to 0. \notag
\end{align}
Applying $-\otimes_{\mathbf{k}}B$ to the minimal injective coresolution above,
we have an injective coresolution 
\begin{align}
0 \to A\otimes_{\mathbf{k}}B \to I^{0}\otimes_{\mathbf{k}}B \to I^{1}\otimes_{\mathbf{k}}B \to \cdots \to I^{m}\otimes_{\mathbf{k}}B \to I^{m+1}\otimes_{\mathbf{k}}B \to 0 \notag
\end{align}
of $A\otimes_{\mathbf{k}}B$ by Lemma \ref{lem-tensor}(2).
We show that if $\pd I^{i}\le 1$, then $\pd (I^{i}\otimes_{\mathbf{k}}B) \le 1$.
Applying $-\otimes_{\mathbf{k}}B$ to a minimal projective resolution $0\to P^{i}_{1}\to P^{i}_{0}\to I^{i}\to 0$ induces a projective resolution of $I^{i}\otimes_{\mathbf{k}}B$
\begin{align}
0\to P^{i}_{1}\otimes_{\mathbf{k}}B \to P^{i}_{0}\otimes_{\mathbf{k}}B \to I^{i}\otimes_{\mathbf{k}}B \to 0. \notag
\end{align}
Hence $A\otimes_{\mathbf{k}}B$ is an almost $n$-Auslander--Gorenstein algebra.

Next we show $\gl (A\otimes_{\mathbf{k}}B)=\infty$.
Since $\gl B =\infty$, there exists a $B$-module $M$ such that $\pd M =\infty$. 
Hence we have $\pd (A\otimes_{\mathbf{k}}M) =\infty$. 
This finishes the proof.
\end{proof}

\begin{remark}
In Proposition \ref{prop-hmky}, we can replace a self-injective algebra with a graded Frobenius algebra whose zeroth part is self-injective by using results in \cite{MY}.
\end{remark}

\section{The endomorphism algebras of the $d$-tilting modules}
In this section, we study the endomorphism algebra $B^{d}:=\End_{A}(\mathbb{T}^{d})$ of the $d$-tilting module $\mathbb{T}^{d}$ over an almost $n$-Auslander algebra $A$.
Throughout this section, $I$ is a basic injective $A$-module with the property that $\add I$ consists of all injective $A$-modules with projective dimension at most one and $A$ is an almost $n$-Auslander algebra.
If $A$ is an almost $0$-Auslander algebra, or equivalently, a hereditary algebra, then $\mathbb{T}^{1}\cong \kD A$ and $B^{1}\cong A$.
In the following, we always assume $\gl A > 1$, that is, $\gl A = n+1 = \gdom{I}A$.
Let $0\to A \xto{f^{0}}I^{0}\xto{f^{1}}I^{1}\to \cdots \xto{f^{n+1}} I^{n+1}\to 0$ be a minimal injective coresolution and $\mathbb{T}^{d}$ the basic module of $\cok f^{d-1} \oplus I$. 
Then $\mathbb{T}^{d}$ is a $d$-tilting $A$-module for all $1\le d\le n+1$.

We start this section with observing the projective dimension of $\Hom_{A}(\mathbb{T}^{d},I')$ for an injective module $I'$. 

\begin{lemma} \label{lem-pd-b} 
Let $I'$ be an indecomposable injective $A$-module.
Then we have 
\begin{align}
\pd \Hom_{A} (\mathbb{T}^d, I') \le 
\begin{cases}
0 & (I'\in \add I),\\
n+1-d & (I'\notin \add I).
\end{cases}\notag
\end{align}
\end{lemma}
\begin{proof}
If $I'\in \add I$, then $\Hom_{A} (\mathbb{T}^{d},I')$ is a projective $B^{d}$-module, and hence the assertion holds.
In the following, we assume $I'\notin \add I$.
Then we have $I'\in \add I^{n+1}$ by Lemma \ref{lem-ARS}(5).
By Proposition \ref{prop-mutation-resol}, we have an exact sequence 
\begin{align}
0 \to X \to I_{n-d} \to \cdots \to I_{0} \to I' \to 0 \notag
\end{align}
with $I_{i} \in \add I$ and $X \in \add \mathbb{T}^{d}$. 
Applying $\Hom_A(\mathbb{T}^{d}, -)$, we have a projective resolution of $\Hom_{A}(\mathbb{T}^{d}, I')$
\begin{align}
0 \to \Hom_{A}(\mathbb{T}^{d}, X) \to \Hom_{A}(\mathbb{T}^{d}, I_{n-d}) \to \cdots \to \Hom_{A}(\mathbb{T}^{d}, I_{0}) \to \Hom_{A}(\mathbb{T}^{d}, I') \to 0. \notag
\end{align}
Thus the proof is complete.
\end{proof}

Let $\nu$ be the Nakayama functor of $\mod A$.
By Lemma \ref{lem-pd-b}, we give an upper bound for global dimension of $B^{d}$.

\begin{proposition}\label{prop-gl-b}
Fix an integer $n \ge 1$.
Assume that $A$ is an almost $n$-Auslander algebra. 
Let 
\begin{align}\label{seq-proj-resol-t}
0 \to P_{d}^{\mathbb{T}^{d}} \to \cdots \to P_{1}^{\mathbb{T}^{d}} \to P_{0}^{\mathbb{T}^{d}} \to \mathbb{T}^{d} \to 0 
\end{align}
be a minimal projective resolution of $\mathbb{T}^{d}$. 
Then the following statements hold.
\begin{itemize}
\item[(1)] $\gl B^{d}\le \gl A$.
\item[(2)] If $\nu P_{1}^{\mathbb{T}^{1}} \in \add I$, then $\gl B^{1} = n$.
\end{itemize}
\end{proposition}
\begin{proof}
(1) Due to \cite[Proposition III.3.4]{H88}, we have
\begin{align}\label{eq-Happel3.4}
|\gl A - \gl B^{d}| \le \pd \mathbb{T}^{d}. 
\end{align}
Hence we obtain $\gl B^{d}<\infty$. 
Thus it is enough to show $\pd \kD B^{d}\le n+1$.
Applying $\Hom_{A}(-,\mathbb{T}^{d})$ to the exact sequence \eqref{seq-proj-resol-t} induces an exact sequence
\begin{align}
0 \to \Hom_{A}(\mathbb{T}^{d}, \mathbb{T}^{d}) \to \Hom_{A}(P_{0}^{\mathbb{T}^{d}}, \mathbb{T}^{d}) \to \cdots \to \Hom_{A}(P_{d}^{\mathbb{T}^{d}}, \mathbb{T}^{d}) \to 0. \notag
\end{align}
By Serre duality, we have $\Hom_{A}(P_{i}^{\mathbb{T}^{d}}, \mathbb{T}^{d}) \cong \kD \Hom_{A}(\mathbb{T}^{d}, \nu P_{i}^{\mathbb{T}^{d}})$. 
Applying $\kD$, we have an exact sequence
\begin{align}
0 \to \Hom_{A}(\mathbb{T}^{d}, \nu P_{d}^{\mathbb{T}^{d}}) \to \cdots \to \Hom_{A}(\mathbb{T}^{d}, \nu P_{0}^{\mathbb{T}^{d}}) \to \kD B^{d} \to 0. \notag
\end{align}
By Lemma \ref{lem-ASS-A47}, we have
\begin{align}\label{seq-inequality}
\pd \kD B^{d} \le \max\{ \pd \Hom_{A}(\mathbb{T}^{d}, \nu P_{i}^{\mathbb{T}^{d}})+i \mid i\in \{ 0,1,\ldots,d \} \}\le n+1,
\end{align}
where the last inequality follows from Lemma \ref{lem-pd-b}.

(2) By \eqref{seq-inequality}, we have 
\begin{align}
\pd \kD B^{1} 
&\le \max\{ \pd \Hom_{A}(\mathbb{T}^{1}, \nu P_{0}^{\mathbb{T}^{1}}), \pd \Hom_{A}(\mathbb{T}^{1}, \nu P_{1}^{\mathbb{T}^{1}})+1 \} \le n, \notag
\end{align}
where the last inequality follows from Lemma \ref{lem-pd-b} since $\nu P_{1}^{\mathbb{T}^{1}} \in \add I$.
Thus $\gl B^{1}\le n$.
On the other hand, by \eqref{eq-Happel3.4}, we have $\gl B^{1}\ge n$. This finishes the proof.
\end{proof}

In the rest of this section, we give a sufficient condition for $B^{\op}$ to be an almost $n$-Auslander algebra again, where $B: = B^{1}= \End_{A}(\mathbb{T}^1)$.
We define $\mathcal{D}$ to be the full subcategory of $\mod A$ consisting of $A$-modules $X$ with $\id \Hom_{A}(\mathbb{T}^1, X) \le 1$.
Note that for each $X\in \Fac_{1}(I)$, we have $\id \Hom_{A}(\mathbb{T}^1, X) \le 1+ \id X$ by \cite[VI.7.20]{ASS06}.
Thus we have $I\in\mathcal{D}$.
The following theorem is a main result of this section.

\begin{theorem} \label{thm-domdim-b}
Fix an integer $n \ge 1$. 
Let $A$ be an almost $n$-Auslander algebra and $B:=\End_{A}(\mathbb{T}^1)$.
Let $I^{\circ}$ be a basic injective $B^{\op}$-module with the property that $\add I^{\circ}$ consists of all injective $B^{\op}$-modules with projective dimension at most one.
Then we have $\gdom{I^{\circ}}B^{\op} \ge n$.  
Moreover if $\mathbb{T}^{1} \in \mathcal{D}$, then $B^{\op}$ is an almost $n$-Auslander algebra.
\end{theorem}

To prove Theorem \ref{thm-domdim-b}, we need the following lemma.

\begin{lemma} \label{lem-clas-domdim-b}
Keep the notation in Theorem \ref{thm-domdim-b}. 
Let $P \in \proj A$. 
Then the following statements hold.
\begin{itemize}
\item[(1)] If $\nu P \in \add I$, then $\Hom_{A}(P, \mathbb{T}^{1})$ is an injective $B^{\op}$-module with projective dimension at most one. 
\item[(2)] $\Hom_{A}(P, \mathbb{T}^{1}) \in \Sub^{n}(I^{\circ})$ holds.
\item[(3)] If $\mathbb{T}^{1}\in \mathcal{D}$, then we have $\Hom_{A}(P, \mathbb{T}^1) \in \Sub^{n+1}(I^{\circ})$.
\end{itemize}
\end{lemma}
\begin{proof}
(1) Since $\nu P \in \add I$, we obtain that $\Hom_{A}(P, \mathbb{T}^1) \cong \kD \Hom_{A}(\mathbb{T}^1, \nu P)$ is injective.
Since $\mathbb{T}^{1}$ is tilting, there exists an exact sequence $0 \to P \to T^0 \to T^1 \to 0$ with $T^{0}, T^{1} \in \add \mathbb{T}^1$. 
Applying $\Hom_{A}(-, \mathbb{T}^{1})$ to this exact sequence, we have an exact sequence
\begin{align}
0 \to \Hom_{A}(T^{1}, \mathbb{T}^{1}) \to \Hom_{A}(T^{0}, \mathbb{T}^{1}) \to \Hom_{A}(P, \mathbb{T}^{1}) \to 0. \notag
\end{align}
Therefore the assertion follows from $\Hom_{A}(T^{1}, \mathbb{T}^{1}), \Hom_{A}(T^{0}, \mathbb{T}^{1}) \in \proj B^{\op}$. 

(2) By Proposition \ref{prop-mutation-resol}, there exists an exact sequence 
\begin{align}
0 \to X \to I_{n-1} \to \cdots \to I_{0} \to \nu P \to 0 \notag
\end{align}
such that $I_{i} \in \add I$ and $X \in \add \mathbb{T}^{1}$.
Applying $\kD\Hom_{A}(\mathbb{T}^1, -)$ to this exact sequence and using Serre duality, we have an exact sequence
\begin{equation}\label{dhom-resol}
0 \to \Hom_A(P, \mathbb{T}^1) \to \Hom_A(P_0, \mathbb{T}^1) \to \cdots \to \Hom_A(P_{n-1}, \mathbb{T}^1) \to \kD \Hom_A(\mathbb{T}^1, X) \to 0, 
\end{equation}
where $I_{i}=\nu P_{i}$ for each $0\le i\le n-1$. 
By (1), $\Hom_{A}(P_{i}, \mathbb{T}^1)$ is injective with projective dimension at most one.
Thus $\Hom_{A}(P, \mathbb{T}^{1}) \in \Sub^{n}(I^{\circ})$. 

(3) Assume $\mathbb{T}^{1}\in \mathcal{D}$.
Since $X\in \add \mathbb{T}^{1}$, we obtain that $\kD\Hom_{A}(\mathbb{T}^{1},X)$ is injective with projective dimension at most one.
Hence \eqref{dhom-resol} implies that $\Hom_{A}(P, \mathbb{T}^{1}) \in \Sub^{n+1}(I^{\circ})$. 
The proof is complete.
\end{proof}

Now we are ready to show Theorem \ref{thm-domdim-b}.
\begin{proof}[Proof of Theorem \ref{thm-domdim-b}]
Note that $\pd \mathbb{T}^{1}= 1$.
Applying $\Hom_{A}(-,\mathbb{T}^{1})$ to a minimal projective resolution $0\to P_{1} \to P_{0} \to \mathbb{T}^{1} \to 0$, we have an exact sequence
\begin{align}
0\to \Hom_{A}(\mathbb{T}^{1},\mathbb{T}^{1}) \to \Hom_{A}(P_{0},\mathbb{T}^{1}) \to \Hom_{A}(P_{1},\mathbb{T}^{1})\to 0. \notag
\end{align}
By Lemma \ref{lem-clas-domdim-b}(2), we have $\Hom_{A}(P_{0},\mathbb{T}^{1}),\Hom_{A}(P_{1},\mathbb{T}^{1})\in \Sub^{n}(I^{\circ})$.
This exact sequence gives $\Hom_{A}(\mathbb{T}^{1},\mathbb{T}^{1})\in \Sub^{n}(I^{\circ})$ by \cite[Lemma 3.1(1)]{CX16}.
Hence we have $\gdom{I^{\circ}}B^{\op} \ge n$.

Now we assume $\mathbb{T}^{1}\in \mathcal{D}$. 
Then $\Hom_{A}(P_{0},\mathbb{T}^{1}),\Hom_{A}(P_{1},\mathbb{T}^{1})\in \Sub^{n+1}(I^{\circ})$ by Lemma \ref{lem-clas-domdim-b}(3).
Due to \cite[Lemma 3.1(1)]{CX16}, we have $\gdom{I^{\circ}}B^{\op} \ge n+1$.
Hence the assertion follows from Proposition \ref{prop-gl-b}(1).
\end{proof}

\section{Almost Auslander algebras and strongly quasi-hereditary algebras}
In this section, we study a relationship between almost $1$-Auslander algebras and strongly quasi-hereditary algebras. 
We start with recalling the definition of strongly quasi-hereditary algebras (see \cite{R10}, \cite{DR92} and \cite{T20} for details).
Let $A$ be an artin algebra.
We fix a complete set $\{ S(\lambda) \mid \lambda \in \Lambda \}$ of representatives of isomorphism classes of simple $A$-modules.
We denote by $P(\lambda)$ the projective cover of $S(\lambda)$ and $I(\lambda)$ the injective hull of $S(\lambda)$.
Let $\le$ be a partial order on $\Lambda$.
For each $\lambda\in\Lambda$, we denote by $\Delta(\lambda)$ the standard $A$-module (that is, it is the maximal factor module of $P(\lambda)$ whose composition factors have the form $S(\mu)$ for some $\mu\le \lambda$). 
Dually, we define the costandard module $\nabla(\lambda)$ for each $\lambda\in \Lambda$.
Let $\mathcal{F}(\Delta)$ be the full subcategory of $\mod A$ whose objects are the modules which have a $\Delta$-filtration.
For $M \in \mathcal{F}(\Delta)$, we denote by $(M: \Delta(\lambda))$ the filtration multiplicity of $\Delta(\lambda)$, which does not depend on the choice of $\Delta$-filtration.

The pair $(A,\le)$ (or simply $A$) is called a \emph{quasi-hereditary algebra} if for each $\lambda\in \Lambda$ there exists an exact sequence
\begin{align}
0 \to K(\lambda) \to P(\lambda) \to \Delta(\lambda) \to 0 \notag
\end{align}
satisfying the following conditions:
\begin{itemize}
\item $\End_{A}(\Delta(\lambda))$ is a division ring;
\item $K(\lambda)\in \mathcal{F}(\Delta)$;
\item if $(K(\lambda):\Delta(\mu))\neq 0$, then $\lambda<\mu$.
\end{itemize}

It is well known that all quasi-hereditary algebras have finite global dimension (see \cite[Theorem 4.3]{PaSc88}).  
By \cite[Theorem 5]{R91}, a quasi-hereditary algebra $A$ has a basic tilting-cotilting $A$-module $T$, which is a direct sum of all Ext-injective objects in $\mathcal{F}(\Delta)$.
We call $T$ the \emph{characteristic tilting module}.

\begin{definition}[{\cite[Proposition A.1]{R10}}]
Let $(A, \le)$ be a quasi-hereditary algebra and $T$ its characteristic tilting module.
\begin{itemize}
\item[(1)] The pair $(A,\le)$ (or simply $A$) is called a \emph{right-strongly quasi-hereditary algebra} if it satisfies one of the following equivalent conditions.
\begin{itemize}
\item[(a)] $\pd \Delta(\lambda)\le 1$ for each $\lambda\in \Lambda$.
\item[(b)] $\pd X\le 1$ for each $X\in\mathcal{F}(\Delta)$.
\item[(c)] $\pd T \le 1$.
\end{itemize}
Dually, we define a left-strongly quasi-hereditary algebra.
\item[(2)] The pair $(A,\le)$ (or simply $A$) is called a \emph{strongly quasi-hereditary algebra} if it is both right-strongly quasi-hereditary and left-strongly quasi-hereditary.
\end{itemize}
\end{definition}

Ringel shows that if $A$ is strongly quasi-hereditary, then its global dimension is at most two (see \cite[Proposition A.2]{R10}).
However, the converse does not hold in general.
On the other hand, if $\gl A \le 2$, then there exists a partial order $\le$ on $\Lambda$ such that $(A,\le)$ is a right-strongly quasi-hereditary algebra but not necessarily strongly quasi-hereditary (see \cite[Theorems 4.1 and 4.6]{T20}).
Then we have the following question.
\begin{question}
Assume that $\gl A \le 2$ so that $(A,\le)$ is right-strongly quasi-hereditary.
When is $(A,\le)$ a strongly quasi-hereditary algebra? 
\end{question}

In \cite{T20} and \cite{T19}, the author gives a complete answer to the question when $A$ is an Auslander algebra or an Auslander--Dlab--Ringel algebra.
In the following, we give a partial answer for almost $1$-Auslander algebras.
We assume that $A$ is an almost $1$-Auslander algebra. 
Let $I$ be a basic injective $A$-module with the property that $\add I$ consists of all injective $A$-modules with projective dimension at most one and $\mathbb{T}^{1}$ the basic $1$-tilting module.
Since $\gl A \le 2$, we can choose $\le$ so that $(A,\le)$ is right-strongly quasi-hereditary, and let $\mathbb{T}$ be its characteristic tilting module. 

The following theorem is a main result of this section.

\begin{theorem}\label{thm-tilting-sqh}
Let $A$ be an almost $1$-Auslander algebra.
Keep the notation above.
Consider the following conditions:
\begin{itemize}
\item[(1)] $(A,\le)$ is a strongly quasi-hereditary algebra;
\item[(2)] $\mathbb{T}\cong \mathbb{T}^{1}$;
\item[(3)] $P(\mathbb{T})\in \add I$, where $P(\mathbb{T})$ is the projective cover of $\mathbb{T}$.
\end{itemize}
Then {\rm(3)}$\Rightarrow${\rm(2)}$\Rightarrow${\rm(1)} holds.
Moreover if $I$ is projective, then {\rm(1)}$\Rightarrow${\rm(3)} holds. 
\end{theorem}

First we give an observation for almost $0$-Auslander algebras or equivalently, hereditary algebras.
\begin{example}\label{example-hereditary}
\begin{itemize}
\item[(1)] Any almost $0$-Auslander algebra is always a strongly quasi-hereditary algebra since all standard modules have projective dimension at most one and all costandard modules have injective dimension at most one.
\item[(2)] Let $A$ be an artin algebra. If $A$ is a right-strongly (respectively, left-strongly) quasi-hereditary algebra with $T\cong\kD A$ (respectively, $T\cong A$), then $A$ is an almost $0$-Auslander algebra. 
Indeed, since $A$ is a right-strongly quasi-hereditary algebra, we have $\pd T \le 1$, and hence $\pd \kD A \le 1$. 
Hence the assertion follows from Lemma \ref{lem-ARS}(1).
\end{itemize}
\end{example}

To prove Theorem  \ref{thm-tilting-sqh}, we need the following lemma.
\begin{lemma}\label{lem-pd-sub}
The following statements hold.
\begin{itemize}
\item[(1)] Let $I$ be an injective $A$-module.
Assume that $A \in \Sub^{2}(I)$.
If $\pd X \le 1$, then the injective hull $I(X)$ is in $\add I$. In particular, $X\in \Sub^{1}(I)$.
\item[(2)] Let $P$ be a projective $A$-module.
Assume that $\kD A \in \Fac_{2}(P)$.
If $\id Y \le 1$, then  the projective cover $P(Y)$ is in $\add P$. In particular, $Y\in \Fac_{1}(P)$.
\end{itemize}
\end{lemma}
\begin{proof}
We only prove (1); the proof of (2) is similar. 
If $\pd X =0$, then $X \in \add A$, and hence the assertion holds. 
We assume $\pd X =1$.
Then we obtain a minimal projective resolution 
\begin{align}
0 \to P_{1} \xto{\rho_{1}} P_{0} \xto{\rho_{0}} X \to 0. \notag
\end{align}
Let $\iota_{i}:P_{i}\to I(P_{i})$ be the injective hull of $P_{i}$ for each $i\in\{ 0,1 \}$.
Then we have the following commutative diagram:
\begin{align}
\xymatrix{
0\ar[r]&P_{1}\ar[r]^{\rho_{1}}\ar[d]^{\iota_{1}}&P_{0}\ar[r]^{\rho_{0}}\ar[d]^{\iota_{0}}&X\ar[r]\ar[d]^{\iota}&0\; \\
0\ar[r]&I(P_{1})\ar[r]^{\rho'_{1}}&I(P_{0})\ar[r]^{\rho'_{0}}&X'\ar[r]&0.
}\notag
\end{align}
Since $\rho_{1}'$ is split and $A \in \Sub^{2}(I)$, we have $X' \in \add I$.
By the Snake lemma, there exists a monomorphism $\ker\iota \to \cok\iota_{1}$. 
Since $A\in \Sub^2(I)$, $\cok \iota_{1}$ is embedded into some $I' \in \add I$, and hence so is $\ker \iota$.
Let $I(\ker \iota)$ and $I(\im \iota)$ be injective hulls of $\ker \iota$ and $\im \iota$ respectively.
Then we obtain that $I(\ker \iota)$ and $I(\im \iota)$ are contained in $\add I$ because there exist monomorphisms $\ker \iota \to I'$ and $\im \iota \to X'$.
Since we have a monomorphism $X \to I(\ker \iota) \oplus I(\im \iota) \in \add I$, the injective hull $I(X)$ of $X$ is contained in $\add I$.
\end{proof}

Now we are ready to prove Theorem \ref{thm-tilting-sqh}.
\begin{proof}[Proof of Theorem \ref{thm-tilting-sqh}]

(3)$\Rightarrow$(2): Since $A$ is right strongly quasi-hereditary, we have $\pd\mathbb{T} \le 1$.
On the other hand, $P(\mathbb{T})\in \add I$ implies $\mathbb{T}\in \Fac_{1}(I)$.
Thus $\mathbb{T}\cong \mathbb{T}^{1}$ by Lemma \ref{lem-uniquness-tilting}(3).

(2)$\Rightarrow$(1): Note that $\mathbb{T}^1$ is a $1$-cotilting module by Theorem \ref{thm-almost-AG}.  
Since $\id \mathbb{T} =\id \mathbb{T}^{1}\le 1$ holds, $A$ is a left-strongly quasi-hereditary algebra.
Hence the assertion holds.

In the following, we assume that $I$ is projective.

(1)$\Rightarrow$(3): Since $I \in \proj A$, we have $\gl A=2$.
By Lemma \ref{lem-cotilting}, there exists an exact sequence $0 \to \mathbb{T}^{1} \to J_{1} \oplus I \to \kD A\to 0$ such that $J_{1} \in \add I$.
Thus we obtain $\kD A \in \Fac_{2}(I)$ because $\mathbb{T}^{1} \in \Fac_{1}(I)$.
Since $A$ is left-strongly quasi-hereditary, we have $\id \mathbb{T} \le 1$. 
By Lemma \ref{lem-pd-sub}(2), the projective cover $P(\mathbb{T})$ is in $\add I$.
\end{proof}

If we do not assume that $I$ is projective, then (1)$\Rightarrow$(2) is not always satisfied as the following examples show.

\begin{example}
\begin{itemize}
\item[(1)] Let $A$ be an almost $0$-Auslander algebra.
Then we have $\mathbb{T}^{1}=\kD A \not \cong A$.
On the other hand, by Example \ref{example-hereditary}(1), $A$ is a strongly quasi-hereditary algebra with characteristic tilting module $\mathbb{T}$. 
If $\mathbb{T}\cong A$, then we have $\mathbb{T}\neq \mathbb{T}^{1}$.
For example, when $A$ is the path algebra of $1\to 2\to 3$ with partial order $\{ 3 < 2< 1 \}$, we have $\mathbb{T}\cong A \not\cong \kD A \cong \mathbb{T}^{1}$.
\item[(2)] Let $A$ be the algebra defined by the quiver
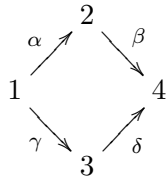
\begin{align} 
\xymatrix@=15pt{ &2 \ar[rd]^{\beta} \\
 1 \ar[ru]^{\alpha} \ar[rd]_{\gamma} && 4 \\
 &3 \ar[ru]_{\delta}
 }
\notag
\end{align}
with a relation $\alpha \beta- \gamma \delta$.
Then we obtain $I=I(2)\oplus I(3)\oplus I(4)$, which is not projective.
Moreover, $\gl A = 2 = \gdom{I} A$ holds. Indeed, $A$ has a minimal injective coresolution
\begin{align}
0 \to A \to I(4) ^{\oplus 4} \to I(2)^{\oplus 2} \oplus I(3)^{\oplus 2} \to I(1) \to 0.\notag
\end{align}
Therefore we have $\mathbb{T}^1= I(4)/S(4) \oplus I(2) \oplus I(3) \oplus I(4)$.
On the other hand, $A$ is a strongly quasi-hereditary algebra with respect to $\{ 2<3<1<4 \}$ and the characteristic tilting module is $\mathbb{T}= I(4)/S(4) \oplus S(2) \oplus S(3) \oplus I(4)$.
\end{itemize}
\end{example}

Let $A$ be an artin algebra with $\gl A=2$. 
Then $A$ is an Auslander algebra if and only if $\gdom{I}A\ge 2$ and $I \in \proj A$. 
Hence, as an application of Theorem \ref{thm-tilting-sqh}, we have the following corollary.

\begin{corollary} \label{cor-sqh-Aus}
Let $A$ be an Auslander algebra and $\mathbb{T}$ the characteristic tilting module of $(A, \le)$. 
Then the following statements are equivalent.
\begin{itemize}
\item[(1)] $(A, \le)$ is a strongly quasi-hereditary algebra.
\item[(2)] $\mathbb{T}\cong \mathbb{T}^{1}$.
\item[(3)] $P(\mathbb{T})\in \add I$, where $P(\mathbb{T})$ is the projective cover of $\mathbb{T}$.
\item[(4)] $\End_{A}(I)$ is a Nakayama algebra.
\end{itemize} 
\end{corollary}
\begin{proof}
(1)$\Leftrightarrow$(2)$\Leftrightarrow$(3): This follows from Theorem \ref{thm-tilting-sqh}. 

(1)$\Leftrightarrow$(4): This follows from \cite[Theorem 4.6]{T20}.
\end{proof}

As an application, we give the following proposition, which is a generalization of \cite[$\S$ 7]{DR92} and \cite{E17}. 

\begin{proposition} \label{prop-dre}
Let $A$ be an Auslander algebra and $eA$ a maximal projective-injective direct summand of $A$.
If $A$ is strongly quasi-hereditary, then $\mod (A/AeA)$ is equivalent to $\mathcal{F}(\Delta)/\add \mathbb{T}^{1}$. 
\end{proposition}

In the rest of this section, we give a proof of Proposition \ref{prop-dre} following the strategy of \cite{DR92}.

\begin{lemma}[{\cite[Theorem 3]{DR92}}] \label{lem-dr-thm3}
Let $A$ be a strongly quasi-hereditary algebra and $T$ the characteristic tilting module of $A$. 
Then we have an equivalence $\mathcal{F}(\Delta)/ \add T \simeq \mathcal{H}(T)$, where $\mathcal{H}(T):= \{ Y \in \mod A \mid \Hom_A(T, Y)=0\}$.
\end{lemma}

For $M, N \in \mod A$, we denote by $\tr_{N} M$ the trace of $N$ in $M$ (that is, it is the submodule of $M$ generated by all homomorphic images of $N$ in $M$). 
 
\begin{lemma}[{\cite[Theorem 4]{DR92}}] \label{lem-dr-thm4}
Assume that $A$ is a quasi-hereditary algebra and the projective cover of every costandard module is injective. 
Then we have $\mathcal{H}(T) = \mod (A/\tr_{T} A)$.
\end{lemma}

Now we are ready to prove Proposition \ref{prop-dre}.

\begin{proof}[Proof of Proposition \ref{prop-dre}]
Since $A$ is a strongly quasi-hereditary Auslander algebra, its characteristic tilting module $\mathbb{T}$ is isomorphic to $\mathbb{T}^{1}$.
First we prove $\tr_{\mathbb{T}} M = \tr_{P(\mathbb{T})} M$ for each $M \in \mod A$. 
Let $\pi: P(\mathbb{T})\to \mathbb{T}$ be the projective cover of $\mathbb{T}$.
Then for each morphism $f:\mathbb{T}\to M$, we have $\im f =\im f\pi$.
Hence $\tr_{\mathbb{T}}M\subseteq\tr_{P(\mathbb{T})}M$.
Conversely, we show $\tr_{\mathbb{T}}M\supseteq\tr_{P(\mathbb{T})}M$.
Since $A$ is an Auslander algebra, $\kD A$ is in $\Fac_{2}(eA)$.
Moreover, by the definition of left-strongly quasi-hereditary algebras, we have $\id \mathbb{T} \le 1$.
By Lemma \ref{lem-pd-sub}(2), the projective cover $P(\mathbb{T})$ of $\mathbb{T}$ is in $\add eA$. 
It follows from Corollary \ref{cor-sqh-Aus} that $P(\mathbb{T})\in \add \mathbb{T}$.
Hence we have the assertion.

Next, we show $\mathcal{F}(\Delta)/\add \mathbb{T} \cong \mod (A/ AeA)$.
By Lemma \ref{lem-dr-thm3}, we have $\mathcal{F}(\Delta)/ \add \mathbb{T} \simeq \mathcal{H}(\mathbb{T})$.
Since $A$ is a left-strongly quasi-hereditary algebra, the injective dimension of each costandard module is at most one. 
By Lemma \ref{lem-pd-sub}(2), the projective cover of every costandard module is injective.
Thus we obtain that $\mathcal{H}(\mathbb{T}) = \mod (A/\tr_{\mathbb{T}} A)$ by Lemma \ref{lem-dr-thm4}.
Since $\add P(\mathbb{T})=\add eA$, we can easily check that $\tr_{\mathbb{T}} A=\tr_{\mathbb{P(\mathbb{T})}}A= AeA$.
Hence we obtain $\mathcal{F}(\Delta)/\add \mathbb{T} \cong \mod (A/ AeA)$.
The assertion follows from Corollary \ref{cor-sqh-Aus}.
\end{proof}

\subsection*{Acknowledgements}
The authors would like to express their deep gratitude to Hiroyuki Minamoto for stimulating discussions and suggesting Proposition \ref{prop-hmky}.
T.A. is grateful to Kota Yamaura for answering various questions about his paper.
The authors are greatly indebted to the referee for his/her valuable comments and pointing out an error in the proof of Lemma \ref{lem-pd-sub}.
The authors would also like to thank Ryoichi Kase for helpful comments.

\end{document}